\documentclass[review,onefignum,onetabnum]{siamart171218}

\usepackage{lipsum}
\usepackage{amsfonts}
\usepackage{graphicx}
\usepackage{epstopdf}
\usepackage{algorithmic}
\usepackage{bm}
\usepackage{longtable}
\usepackage{multirow}
\usepackage{array}
\newcolumntype{P}[1]{>{\centering\arraybackslash}p{#1}}
\ifpdf
  \DeclareGraphicsExtensions{.eps,.pdf,.png,.jpg}
\else
  \DeclareGraphicsExtensions{.eps}
\fi

\newsiamremark{remark}{Remark}
\newsiamremark{hypothesis}{Hypothesis}
\crefname{hypothesis}{Hypothesis}{Hypotheses}
\newsiamthm{claim}{Claim}
\newsiamthm{example}{Example}

\usepackage{amsopn}
\usepackage{mathrsfs}

\makeatletter
\newcommand*{\addFileDependency}[1]{% argument=file name and extension
  \typeout{(#1)}% latexmk will find this if $recorder=0 (however, in that case, it will ignore #1 if it is a .aux or .pdf file etc and it exists! if it doesn't exist, it will appear in the list of dependents regardless)
  \@addtofilelist{#1}% if you want it to appear in \listfiles, not really necessary and latexmk doesn't use this
  \IfFileExists{#1}{}{\typeout{No file #1.}}% latexmk will find this message if #1 doesn't exist (yet)
}
\makeatother

\headers{On Generalized Gauss Quadratures for M\"{u}ntz Systems}{Huaijin Wang \& Chuanju Xu}

\title{On Weighted Generalized Gauss Quadratures for M\"{u}ntz Systems\thanks{This research is partially supported by NSFC grant 11971408.}}
\author{Huaijin Wang\thanks{School of Mathematical Sciences, Xiamen University, 361005 Xiamen, China.}
\and Chuanju Xu\thanks{Corresponding author, School of Mathematical Sciences, Xiamen University, 361005 Xiamen, China. Email: \email{cjxu@xmu.edu.cn}.}
}

\begin{document}

\maketitle

% REQUIRED
\begin{abstract}
A novel recurrence formula for moments with respect to M\"{u}ntz-Legendre polynomials is proposed and applied to construct a numerical method for solving generalized Gauss quadratures with power function weight for M\"{u}ntz systems. These quadrature rules exhibit several properties similar to the classical Gaussian quadratures for polynomial systems, including positive weights, rapid convergence, and others. They are applicable to a wide range of functions, including smooth functions and functions with endpoint singularities, commonly found in  integral equations with singular kernels, complex analysis, potential theory, and other areas.
\end{abstract}

% REQUIRED
\begin{keywords}
  Numerical integration, Generalized Gauss quadrature, Moments recurrence, M\"{u}ntz polynomials
\end{keywords}

% REQUIRED
%\begin{AMS}
%  68Q25, 68R10, 68U05
%\end{AMS}

\section{Introduction}
Let $\omega(x)$ be a continuous function that is positive almost everywhere in the interval $[a,b]$. Let $\varphi(x)$ be an integrable function on $[a,b]$. If $\varphi(x)\omega(x)$ is integrable, we refer to its integral as the integral of $\varphi(x)$ with respect to the weight function $\omega(x)$, denoted by
\begin{equation}
I[\varphi]  = \int_a^b \varphi(x) \omega(x) \mathrm{d} x.
\label{symbol_integral}
\end{equation}
According to the definition of integration, we can select certain nodes $x_k$ on the interval $[a,b]$, and then approximate the value of $I[\varphi]$ through a weighted average of $\varphi(x_k)$. This leads to a numerical quadrature with the following form:
\begin{equation}
Q_N[\varphi]  = \sum_{k=0}^{N} \varphi(x_k) w_k ,
\label{quadrature1}
\end{equation}
where $x_k \in [a,b]$ and $w_k \in \mathbb{R}$ for $k=0,1,\cdots,N$. $x_k$ are called the nodes of the quadrature \cref{quadrature1}, and $\omega_k$ are the corresponding weights.  The selection of $\{\omega_k\}_{k=0}^N$ depends only on $\{x_k\}_{k=0}^N$ and not on the specific form of the integrand $\varphi(x)$.

We use $Q_N[\varphi]$ as an approximation of $I[\varphi]$.
When the nodes $x_k$ are chosen as the zeros of orthogonal polynomials of degree $N+1$ with respect to $\omega(x)$, and the weights $\omega_k$ are determined by the corresponding interpolation polynomials of the nodes, the resulting numerical quadrature $Q_N[\cdot]$, known as Gauss quadrature, achieves the highest algebraic accuracy of $2N+1$.

Gauss quadrature offers several advantages. Firstly, all quadrature nodes $x_k$ are located in the interior of $[a,b]$, and the weights $\omega_k$ are all positive, which ensures the stability of integration computations \cite{li2001}.  Secondly, the Gaussian rule with $N+1$ points is exact for polynomials of degree up to $2N+1$, resulting in a rapid convergence for integrands that can be well approximated by polynomials.  Moreover, Gaussian quadrature can be computed efficiently due to its connection with orthogonal polynomials, with computational costs scaling as $O(N^2)$ for the classic Golub-Welsch algorithm \cite{golub1969calculation} or $O(N)$ for more specialized methods \cite{glaser2007fast}.

Although Gaussian rule can converge for some singular integrands, such as the $\log$ function or M\"{u}ntz polynomials \cite{borwein1994muntz}, which are commonly encountered in the boundary element method and the finite element method for solving partial differential equations, the convergence rate is typically low and using large nodes in quadrature is not recommended. While monomial transformation \cite{lombardi2009}, graded meshes \cite{schwab1994}, or adaptive methods \cite{cools2003} can be used to address these issues, these methods often lack the stability, rapid convergence, and elegance of Gaussian rule.

Gaussian rule can be extended in a natural way to general functions. Let 
\begin{equation}
\label{functions_space}
\left \{ \varphi_0, \varphi_1, \cdots, \varphi_{2N+1} \right  \}
\end{equation}
be a system of linearly independent functions that are usually chosen to be complete in some suitable space of functions \cite{karlin1966,ma1996generalized}. The generalized Gauss quadrature aims to select $x_k$ and $\omega_k$, $k=0,1,\cdots,N$, such that 
\begin{equation}
Q_N[\varphi_j] = I [\varphi_j],\quad j=0,1,\cdots,2N+1.
\label{nolinear_equation}
\end{equation}

The work of Rokhlin and others \cite{ma1996generalized,rokhlin1998generalized,rokhlin1999nonlinear,rokhlin2010nonlinear} has explored the use of generalized Gauss quadrature in numerical algorithms. The method proposed in \cite{ma1996generalized} constructs a mapping from the Hermite system and identifies the zeros of this mapping as the Gaussian nodes. This approach involves a variant of Newton's algorithm coupled with a continuation scheme. Subsequent modifications have been proposed in \cite{rokhlin1998generalized,rokhlin1999nonlinear,rokhlin2010nonlinear} by incorporating preprocessing steps to enhance the robustness of the Gauss quadrature procedures. While these approaches can be used for almost any basis set, they may not be efficient when \cref{functions_space} is restricted to a specific system, such as the M\"{u}ntz system, partly due to the difficulty in estimating the value of this mapping.

Specifically, when constraining the functions in \cref{functions_space} to trigonometric functions \cite{milovanovic2008trigonometric} or spline functions \cite{chen2021explicit}, the corresponding Gauss quadratures have been considered. For the M\"{u}ntz system, a straightforward approach for computing Gauss quadrature was proposed in \cite{milovanovic2005gaussian}. The proposed method utilizes a continuation technique and Newton's method to solve a set of $2N+2$ nonlinear equations \cref{nolinear_equation} for the $2N+2$ unknowns $x_k$ and $\omega_k$. However, the proposed method has some limitations. For instance, when considering a specific M\"{u}ntz sequence $\{\lambda_0,\lambda_1,\cdots,\lambda_{2N+1} \}$ and choosing $\varphi_j$ as the corresponding M\"{u}ntz-Legendre polynomials:
\begin{enumerate}
\item The accuracy of the estimation of $\varphi_j(x)$ may be compromised when $\varphi_j(x)$ is singular at $0$ and $x$ approaches $0$.
\item The method relies on calculating $I[\varphi_j]$ via orthogonality, which can only be done by choosing $\omega(x) = x^{\lambda_0}$.
\item The Jacobian in the Newton equation may not be accurately computed when $\varphi_j(x)$ becomes singular at $0$ and $x$ is close to $0$.
\end{enumerate}

In this paper, we propose a strategy for determining stable parameters for computing M\"{u}ntz-Legendre polynomials, and optimize this process to minimize computational cost using dynamic programming. Subsequently, a recurrence formula (\cref{thm:moment_recurrence}) for the moments of M\"{u}ntz-Legendre polynomials is presented, which enables efficient computation of these moments for arbitrary power weight functions. Finally, we introduce a modification to the original Newton equation by incorporating a damping factor for the step size, with the goal of improving stability and efficiency in the computation process.

This paper is organized as follows. In \cref{sec:preliminaries}, we provide a restatement of some preliminaries. In \cref{sec:polys}, we introduce the concept of orthogonal M\"{u}ntz polynomials and derive some of their recurrence formulae. The detailed computation of M\"{u}ntz-Legendre polynomials will be presented in \cref{sec:muntz_computing}. The numerical construction of the Gauss quadrature rule will be discussed in \cref{sec:quadrature}. Finally, we will address the error estimation of the Gauss quadrature for a specific class of M\"{u}ntz sequences in \cref{sec:error}, and present numerical examples in \cref{sec:examples}.

\section{Preliminaries}
\label{sec:preliminaries}
In this section, we summarize several classical results and numerical tools. They can be found, for example, in \cite{karlin1966,ma1996generalized,allgower2012numerical}.
\subsection{Existence and uniqueness of quadrature rule}
\begin{definition}
[\cite{karlin1966}]
\label{chebyshev_system_def}
A finite set of functions ${\varphi_0,\varphi_1,\cdots,\varphi_n}$ defined on the interval $[a,b]$ is called a Chebyshev system if and only if $\varphi_j\in C[a,b]$, $j=0,1,\cdots,n$, and the determinant of matrix $\bm{\Phi}$ is non-zero, where
\begin{equation}
\bm{\Phi} = \left[
% \begin{array}{cccc}
\begin{matrix}
\varphi_0\left(x_0\right) & \varphi_0\left(x_1\right) & \cdots & \varphi_0\left(x_n\right) \\
\varphi_1\left(x_0\right) & \varphi_1\left(x_1\right) & \cdots & \varphi_1\left(x_n\right) \\
\vdots & \vdots & & \vdots \\
\varphi_n\left(x_0\right) & \varphi_n\left(x_1\right) & \cdots & \varphi_n\left(x_n\right)
% \end{array}
\end{matrix}
\right],
\label{chebyshev_system}
\end{equation}
and $x_0,x_1,\cdots,x_n$ are any distinct points in the interval $[a,b]$.
\end{definition}

\begin{theorem}
[\cite{karlin1966}]
\label{generalized_gauss_quadrature1}
Let $\varphi_0,\varphi_1,\cdots,\varphi_{2N+1}$ be a Chebyshev system of functions defined on the interval $[a,b]$. Then, there exist unique $N+1$ Gaussian nodes $x_k \in (a,b)$ and weights $w_k >0$, $k=0,1,\cdots, N$, such that
\[I[\varphi_j] = Q_N[\varphi_j],\quad  j=0,1,\cdots,2N+1.\]
\end{theorem}

The result of \cref{generalized_gauss_quadrature1} is a corollary of a more general geometric property of moment spaces, as derived from the Chebyshev system \cite{karlin1966}.
Theorem $\ref{generalized_gauss_quadrature1}$ can be extended to a class of Gauss quadratures involving functions with endpoint singularities.
\begin{theorem}
[\cite{ma1996generalized}]
\label{generalized_gauss_quadrature2}
Let $\varphi_j:(a,b)\to \mathbb{R}$ be continuous and integrable on $(a,b)$ for $j=0,1,\cdots,2N+1$, and $r(x)>0$ be continuous and integrable on $(a,b)$. Define $\psi_j(x) = {\varphi_j(x)}/{r(x)}$. If
\begin{equation}
\lim _{x \rightarrow a} \psi_j(x)<\infty,
\label{temp37}
\end{equation}
and ${\psi_0,\psi_1,\cdots,\psi_{2N+1} }$ form a Chebyshev system on the interval $[a,b]$,  and $\omega(x) r(x)$ is integrable, then there exists a unique Gaussian rule with $N+1$ nodes that is exact for ${\varphi_0,\varphi_1,\cdots,\varphi_{2N+1}}$, and all the Gaussian weights $w_0,w_1,\cdots,w_N$ are positive.
\end{theorem}

\subsection{Continuation method}
The continuation method \cite{allgower2012numerical}, also known as the homotopy continuation method, is a numerical technique used to find roots of nonlinear equations. It is a path-following algorithm that traces a solution curve in parameter space from a known solution to the desired one.

Stated briefly, suppose that we are trying to solve a system of $m$ nonlinear equations with $m$ unknows,
\begin{equation}
\label{nonlinear_equation}
\mathbf{F}(\mathbf{x}) = \mathbf{0},
\end{equation}
where $\mathbf{F} : \mathbb{R}^m \rightarrow \mathbb{R}^m$ is continuously differentiable. 
Certainly, if a good approximation $\mathbf{x}^{(0)}$ of the zero point $\mathbf{x}^*$ of $\mathbf{F}$ is available, it is advisable to calculate $\mathbf{x}^*$ by a Newton's method defined by an iteration formula as
\begin{equation}
\label{iteration}
\mathbf{x}^{(k+1)} \leftarrow \mathbf{x}^{(k)} - \mathbf{J}_k^{-1} \mathbf{F}(\mathbf{x}^{(k)}),
\end{equation}
where $\mathbf{J}_k$ is the Jacobian of $\mathbf{F}$ at $\mathbf{x}^{(k)}$. In many cases, the iteration \cref{iteration} fails if a starting point is not available directly.
As a possible remedy, we define a homotopy mapping $\mathbf{H}:\mathbb{R}^m \times \mathbb{R} \rightarrow \mathbb{R}^m$ such that 
\[
\mathbf{H}(\mathbf{x},0) = \mathbf{G}(\mathbf{x}), \quad 
\mathbf{H}(\mathbf{x},1) = \mathbf{F}(\mathbf{x}),
\] 
where $\mathbf{G}:\mathbb{R}^m\rightarrow \mathbb{R}^m$ is a trivial mapping, so that the solution $\mathbf{x}_0$ of $\mathbf{G}(\mathbf{x})=\mathbf{0}$ is unique and known. 
We attempt to trace an implicitly defined curve $ \mathcal{C} \in \mathbf{H}^{-1}(\mathbf{0})$ from a starting point $(\mathbf{x}_0,0)$ to the solution point $(\mathbf{x}^*,1)$.

Suppose that for any $\alpha \in [0,1]$, the system of equations $\mathbf{H}(\mathbf{x},\alpha) = \mathbf{0}$ has a unique solution denoted as $\mathbf{x}(\alpha)$. Furthermore, assume that $\mathbf{x}(\alpha)$ is continuously differentiable with respect to $\alpha$, and that the Jacobian matrix $\partial_{\mathbf{x}} \mathbf{H}$ evaluated at $\mathbf{x}(\alpha)$ is nonsingular.

Then, there exists a partition $0 = \alpha_0 < \alpha_1 < \cdots < \alpha_p = 1$ such that the maximum interval length $\max_{1 \leq i \leq p} {|\alpha_i - \alpha_{i-1}| }$ is sufficiently small.  We can obtain $\mathbf{x}(\alpha_i)$ by solving $\mathbf{H}(\mathbf{x},\alpha_i)=\mathbf{0}$ with initial guess $\mathbf{x}(\alpha_{i-1})$. This process is repeated starting from $\mathbf{x}(\alpha_0)$, and incrementing $i$ from $1$ to $p$. The solution $\mathbf{x}^*$ is obtained when $\alpha_p$ is reached.

A detailed discussion of the continuation method can be found in \cite{allgower2012numerical}, where the convergence of the continuation method under more general conditions is considered.

\section{Orthogonal M\"{u}ntz polynomials}
\label{sec:polys}
Let a complex M\"{u}ntz sequence $\Lambda = \{\lambda_0,\lambda_1,\lambda_2,\cdots \}$, and its first $N+1$ elements be denoted as $\Lambda_N = \{\lambda_0, \lambda_1, \dots, \lambda_N\}$. The set of functions generated by the M\"{u}ntz sequence, $\{x^{\lambda_0},x^{\lambda_1},\cdots,x^{\lambda_N} \}$, is called the M\"{u}ntz system.
The linear space over the field of real numbers generated by the M\"{u}ntz sequence, denoted by $
M(\Lambda_N) = \operatorname{span}\left \{x^{\lambda_0}, x^{\lambda_1}, \cdots, x^{\lambda_N}\right \}
$, is called the M\"{u}ntz space. That is, the M\"{u}ntz space is the collection of all M\"{u}ntz polynomials 
$$
p(x)=\sum_{i=0}^N a_i x^{\lambda_i},\quad a_i \in \mathbb{R}
.$$

For the $L^2$ theory of M\"{u}ntz system \cite{almira2007muntz}, we consider $\Lambda_N$ satisfying 
\begin{equation}
\operatorname{Re}({\lambda_n}) >-\frac{1}{2},\quad n=0,1,2\cdots,N,
\label{lambda_N}
\end{equation}
where $\operatorname{Re}({\lambda_n})$ is the real part of $\lambda_n$. \cref{lambda_N} ensures that every M\"{u}ntz polynomial in $M(\Lambda_N)$ is in $L^2(0,1)$. The corresponding M\"{u}ntz-Legendre polynomials \cite{taslakyan1984}, orthogonal with respect to Lebesgue measure, are defined as 
\begin{equation}
L_n\left(x;\Lambda_n \right) =\frac{1}{2 \pi i} \int_{\Gamma_n}W_n(t) x^t \mathrm{d} t,
\quad
W_n(t) = \prod_{k=0}^{n-1} \frac{t+\bar{\lambda}_k+1}{t-\lambda_k} \frac{1}{t-\lambda_n},
\label{muntz_legendre}
\end{equation}
where $n=0,1,2, \cdots,N$, $\Gamma_n$ is a simple contour on the complex plane that encircles $\lambda_0,\cdots,\lambda_n$,  and $\bar{\lambda}_k$ denotes the conjugate of $\lambda_k$. It is obvious that $L_0(x;\Lambda_0)=x^{\lambda_0}$. Henceforth, without causing confusion, we will denote $L_n (x;\Lambda_n)$ as $L_n (x)$ or simply $L_n$. The definition of $\cref{muntz_legendre}$ does not require the distinctness of the exponents $\lambda_k$. In fact, repeated $\lambda_k$ results in $\log x$ coming into the picture.

\begin{lemma}
\label{repeated_indices}
Let $\Lambda_n=\{\lambda_0,\cdots,\lambda_0,\cdots,\lambda_s,\cdots,\lambda_s\}$ consist of $r_k$ copies of $\lambda_k$, $r_k\geq 1$, $k=0,1,\cdots,s$, $s\leq n$, and $\sum_{k=0}^s r_k = n+1$. Then
\[\text{\rm span} \{L_0,\cdots,L_n\} = \text{\rm span} 
\left \{x^{\lambda_0},\cdots,x^{\lambda_0}\log^{r_0-1}x,\cdots, x^{\lambda_s},\cdots,x^{\lambda_s}\log^{r_s-1}x\right \}.
\]
\end{lemma}

The proof of \cref{repeated_indices} requires \cref{lem:factor}.
\begin{lemma}[\cite{joseph2013}]
\label{lem:factor}
Let $P(x)$ and $Q(x)$ be coprime polynomials with complex coefficients, where the degree of $P(x)$ is lower than the degree of $Q(x)$. Since irreducible polynomials over the field of complex numbers are linear, we suppose the irreducible factorization of $Q(x)$ as follows
\[
Q(x) = \prod_{\nu=0}^s (x-\lambda_\nu)^{r_\nu} ,
\]
where $\lambda_\nu \in \mathbb{C}$ and $r_\nu \in \mathbb{Z}^+$. Then, $P(x)/Q(x)$ can be represented as the sum of simple fractions in the following form
\[
\frac{P(x)}{Q(x)} = \sum_{\nu=0}^s \sum_{k=1}^{r_\nu} \frac{A_{\nu,k}}{(x-\lambda_\nu)^{k}},
\]
where $A_{\nu,k}$ are complex constants for $0 \leq \nu\leq s$ and $1\leq k \leq r_\nu$.
\end{lemma}

\begin{proof}[Proof of \cref{repeated_indices}]
By \cref{lem:factor}, $W_n(t)$ can be factored into
\begin{equation}
\label{temp1005}
W_n(t)
= 
\frac{1}{t+\bar{\lambda}_s+1}
\prod_{\nu=0}^s
\frac{(t+\bar{\lambda}_v+1)^{r_v}}{  (t-\lambda_v)^{r_v}}
=
\sum_{v=0}^{s} \sum_{k=1}^{r_v} \frac{A_{v,k}}{(t-\lambda_v)^k}.
\end{equation}
By multiplying $(t-\lambda_v)^{r_v}$ on both sides of \cref{temp1005} and taking the $(r_v-k)$-th derivative, we can compute $A_{v,k}$ as follows
\[A_{v,k} = 
\frac{1}{(r_v-k)!}
\lim_{t\to \lambda_v}
\partial_t^{r_v-k} 
\left(\frac{\prod_{j=0}^s (t+\bar{\lambda}_j+1)^{r_j}}{(t+\bar{\lambda}_s+1)\prod_{j=0,j\neq v}^s (t-\lambda_j)^{r_j}} \right).
\]
Then
\begin{equation}
\label{temp999}
L_n(x) = \frac{1}{2\pi i} \int_{\Gamma_n}
\sum_{v=0}^{s} \sum_{k=1}^{r_v} \frac{A_{v,k}}{(t-\lambda_v)^k} x^t \mathrm{d}t
 = \sum_{v=0}^{s} \sum_{k=1}^{r_v} A_{v,k}
\frac{1}{2\pi i} \int_{\Gamma_n}
\frac{x^t}{(t-\lambda_v)^k}  \mathrm{d}t
\end{equation}
By the residue theorem \cite{stein2010complex}, \cref{temp999} could be represented as
\[
L_n(x) = 
\sum_{v=0}^{s} \sum_{k=1}^{r_v}
\frac{A_{v,k}}{\Gamma(k)} x^{\lambda_v} \log^{k-1} x .
\]
Therefore, it follows that 
\[\text{\rm span} \{L_0,\cdots,L_n\} \subseteq \text{\rm span} 
\left \{x^{\lambda_0},\cdots,x^{\lambda_0}\log^{r_0-1}x,\cdots, x^{\lambda_s},\cdots,x^{\lambda_s}\log^{r_s-1}x\right \}.
\] 
By a dimension argument, the conclusion is proved.
\end{proof}

To describe the case of repeated elements in the M\"{u}ntz sequence $\Lambda_N$, we define
\begin{equation}
\hat{M}(\Lambda_n) = \mathrm{span} \{ L_0,L_1,\cdots,L_n \},\quad n=0,1,\cdots,N.
\label{muntz_space_hat}
\end{equation}
Thus, for any M\"{u}ntz sequence $\Lambda_N$ that satisfies \eqref{lambda_N}, we always have
\begin{equation}
M(\Lambda_n) \subseteq \hat{M}(\Lambda_n), \quad n=0,1,\cdots,N.
\end{equation}

The name M\"{u}ntz-Legendre polynomial is justified by the following theorem, where the orthogonality of $L_n$ allows repeated indices.
\begin{theorem}[\cite{borwein1994muntz}]
\label{muntz_orhogonal}
Let M\"{u}ntz sequence $\Lambda_N $ satisfy $\cref{lambda_N}$ and $L_n$ be defined by \cref{muntz_legendre}. Then
\begin{equation}
\int_0^1L_n(x) \overline{L_m(x)} \mathrm{d} x = \frac{\delta_{n,m}}{1+\lambda_n + \bar{\lambda}_{n}},\quad n,m=0,1,\cdots,N,
\label{muntz_orthogonal1}
\end{equation}
where $\delta_{n,m}$ is the Kronecker-Delta symbol.
\end{theorem}

Furthermore, we have 
\begin{equation}
\label{value_at_1}
L_n(1) = 1,\quad  L^\prime_n(1) = \lambda_n+\sum_{k=0}^{n-1} (\lambda_k+\bar{\lambda}_k+1),\quad n=0,1,\cdots,N,
\end{equation}
and the recurrence formula with derivatives for $L_n$,
\begin{equation}
x L_n^{\prime}(x)-x L_{n-1}^{\prime}(x)=\lambda_n L_n(x)+\left(1+\bar{\lambda}_{n-1}\right) L_{n-1}(x), \quad n = 1,2,\cdots,N.
\label{muntz_recursive1}
\end{equation}

Moreover, the moments with respect to $L_n$ in the interval $[0,1]$ have the recurrence formula.
\begin{theorem}
\label{thm:moment_recurrence}
Let M\"{u}ntz sequence $\Lambda_N$ and $\lambda \in \mathbb{C}$ satisfying 
\[ \operatorname{Re}(\lambda+\lambda_n) > -1,\quad n=0,1,\cdots,N. \]
M\"{u}ntz-Legendre polynomial $L_n$ is defined by \cref{muntz_legendre}. Then the moments satisfy
\begin{equation}
\label{moment_recurrence}
\int_0^1 L_n(x) x^\lambda \mathrm{d} x = \frac{\lambda-\bar{\lambda}_{n-1}}{1+\lambda+\lambda_n} \int_0^1 L_{n-1}(x) x^\lambda \mathrm{d} x, \quad n\geq 1,
\end{equation}
and $\int_0^1 L_0(x) x^\lambda \mathrm{d} x = {1}/{(1+\lambda+\lambda_0)}$.
\end{theorem}
\begin{proof}
The case of $n=0$ is trivial and we subsequently consider $n\geq 1$. 
We assume $\lambda \notin \Lambda_N$, otherwise there is a limit argument since 
\[\int_0^1 L_n(x) x^\lambda \mathrm{d} x,\quad  n=0,1,\cdots,N, \]
are continuous with respect to $\lambda$.
 Let $\Gamma_n$ encircle $\lambda_0,\lambda_1,\cdots,\lambda_n$ and satisfy
\[\operatorname{Re}(t+\lambda)>-1,\quad \forall t \in \Gamma_n.\]
By \cref{muntz_legendre},
\[
\int_0^1 L_n(x) x^\lambda \mathrm{d} x = \frac{1}{2\pi i} \int_0^1 \int_{\Gamma_n} W_n(t) x^{t+\lambda} \mathrm{d} t \mathrm{d} x.
\]
For any $\varepsilon>0$, there exists $0<\delta < 1$, such that
\[
\left | \int_0^\delta x^{t+\lambda} \mathrm{d} x \right | 
< \varepsilon, \quad \forall t\in \Gamma_n.
\]
Since $W_n(t)$ is continuous on the closed curve $\Gamma_n$, $W_n(t)$ is bounded and achieves its upper and lower bounds on $\Gamma_n$.
Therefore, the integral $\int_0^1 W_n(t) x^{t+\lambda} \mathrm{d} x$ converges uniformly for $t\in \Gamma_n$. Thus, the order of integration can be exchanged, i.e., 
\[
\int_0^1 \int_{\Gamma_n} W_n(t) x^{t+\lambda}  \mathrm{d} t\mathrm{d} x
=
\int_{\Gamma_n} \int_0^1 W_n(t) x^{t+\lambda}  \mathrm{d} x \mathrm{d} t.
\]
Thus,
\[
\begin{aligned}
\int_0^1 L_n(x)x^\lambda \mathrm{d} x
=\frac{1}{2\pi \mathrm{i}}
\int_{\Gamma_n}  \frac{W_n(t)}{t+\lambda+1} \mathrm{d} t.
\end{aligned}
\]
Note that $\Gamma_n$ does not encircle the singularity $-\lambda-1$. We change the integration path to $|t|=R$, where $R>\max \{|\lambda_0|+1,|\lambda_1|+1,\cdots,|\lambda_n|+1,|\lambda|+1 \}$. Then, using the Cauchy integral formula \cite{stein2010complex}, we have
\begin{equation}
\int_0^1 L_n(x)x^\lambda \mathrm{d} x
= \frac{1}{2\pi i} \int_{|t|=R}  \frac{W_n(t)}{t+\lambda+1} \mathrm{d} t - W_n(-\lambda-1).
\label{muntz_integral1}
\end{equation}
Since
\[
\left| \int_{|t|=R} \frac{ W_n(t) }{t+\lambda+1} \mathrm{d} t\right | 
= 
\left |
\int_0^{2\pi}  \frac{W_n(Re^{i\theta}) }{Re^{i\theta}+\lambda+1} Rie^{i\theta} \mathrm{d} \theta
\right | = O\left (\frac{1}{R} \right ),
\]
we take the limit $R\to \infty$ in \cref{muntz_integral1}, then the integral term over $|t|=R$ vanishes, i.e.,
\begin{equation}
    \int_0^1 L_n(x) x^\lambda \mathrm{d} x = - W_n(-\lambda-1).
  \label{temp65}
\end{equation}
Similarly, we have
\begin{equation}
    \int_0^1 L_{n-1}(x) x^\lambda \mathrm{d} x = - W_{n-1}( -\lambda-1).
    \label{temp64}
\end{equation}
By the definition of $W_n(t)$, it follows that 
\begin{equation}
W_n(t)
= W_{n-1}(t) \frac{t+\bar{\lambda}_{n-1}+1}{t-\lambda_n} .
\label{temp1000}
\end{equation}
By combining \cref{temp65}, \cref{temp64}, and \cref{temp1000}, we  obtain \cref{moment_recurrence}.
\end{proof}

Let $\beta \in \mathbb{R}$ and $\Lambda_N$ satisfy 
\begin{equation}
\operatorname{Re}(\lambda_n)+\beta/2 > -1/2, \quad n=0,1,\cdots,N.
\label{lambda_n}
\end{equation}
Putting $\lambda_k + \beta/2$ instead of $\lambda_k$, $k=0,1,\cdots,N$, in the M\"{u}ntz sequence $\Lambda_N$, we can define a kind of M\"{u}ntz-Jacobi polynomials $L_n^\beta (x;\Lambda_n) = x^{-{\beta}/{2}} L_n(x;\Lambda_n+{\beta}/{2})$. Then
\begin{equation}
L_{n}^\beta (x;\Lambda_n) =
\frac{x^{-{\beta}/{2}}}{2\pi i}
\int_{\Gamma} \prod_{k=0}^{n-1} \frac{t+\bar{\lambda}_k + {\beta}/{2} + 1}{t - \lambda_k - {\beta}/{2}}  \frac{x^t}{t-\lambda_n - {\beta}/{2}} \mathrm{d} t.
\label{weighted_muntz1}
\end{equation}
Owing to the properties of M\"{u}ntz-Legendre polynomials, the following results hold.
\begin{theorem}
Let $\beta \in \mathbb{R}$, and $\Lambda_{N}$ satisfy \cref{lambda_n}. Then
\begin{equation}
\int_0^1 L_{n}^\beta(x;\Lambda_n) L_{m}^\beta (x;\Lambda_m) x^\beta \mathrm{d} x = \frac{\delta_{n,m}}{\lambda_n+\bar{\lambda}_n+\beta+1},\quad n,m=0,1,\cdots,N.
\end{equation}
\end{theorem}

It is evident that $\hat{M}(\Lambda_{2N+1}) = \mathrm{span} \{ L_0^\beta,\cdots,L_{2N+1}^\beta \}$. 
Let $\partial_x L_n^\beta$ denote the derivative of $L_n^\beta$.
It follows from \cref{value_at_1} that 
\begin{equation}
L_n^\beta(1) = 1,\quad 
\partial_x L_n^\beta(1) = \lambda_n + \sum_{k=0}^{n-1} (\lambda_k + \bar{\lambda}_k+\beta+1),\quad n=0,1,\cdots,N.
\end{equation}
And the recurrence formula of $L_n^\beta$
\begin{equation}
x \partial_x L_n^\beta - x \partial_x L_{n-1}^\beta = \lambda_{n} L_n^\beta + (1+\bar{\lambda}_{n-1} + \beta) L_{n-1}^\beta.
\end{equation}
Moreover, the moments recurrence read
\begin{equation}
\label{moments_recurrence}
\int_0^1 L_n^\beta x^\beta \mathrm{d} x = \frac{-\lambda_{n-1}}{1+\lambda_n+\beta}
\int_0^1 L_{n-1}^\beta x^\beta \mathrm{d} x.
\end{equation}

\section{Computation of M\"{u}ntz-Legendre polynomials}
\label{sec:muntz_computing}
Let $\Lambda_N$ be a M\"{u}ntz sequence and the corresponding M\"{u}ntz-Legendre polynomial be defined as in \cref{muntz_legendre}.
In contrast to the computation of algebraic polynomials, which can be efficiently achieved using a three-term recurrence formula, M\"{u}ntz-Legendre polynomials lack a recurrence formula that enables high-efficiency computation. Therefore, we resort to the complex integration method \cite{milovanovic1999muntz, milovanovic2017computing}, which aims to directly compute the original definition given in \cref{muntz_legendre}.

\begin{theorem}
[\cite{milovanovic1999muntz}]
For all $x\in (0,1)$, $L_n(x)$ can be written as
\begin{equation}
L_n(x)=
-\frac{x^{\sigma}}{2\pi} \left (
\int_{0}^{+\infty} W_n(\sigma - it) e^{i\omega t} \mathrm{d} t
+ \int_{0}^{+\infty} W_n(\sigma + it) e^{-i\omega t} \mathrm{d} t
\right ), 
\end{equation}
where $\sigma < \min \{ \operatorname{Re}(\lambda_0),\operatorname{Re}(\lambda_1),\cdots,\operatorname{Re}(\lambda_N)  \}$, $\omega = -\log x$.
\end{theorem}

We only consider the real M\"{u}ntz sequence. It is obvious that $\overline{ W_n(t)} = W_n(\bar{t})$ and 
\begin{equation}
L_n(x) = 
-\frac{x^{\sigma}}{\pi}\operatorname{Re}\left \{ \int_0^{+\infty}{W_n(\sigma-it) e^{i\omega t}} \mathrm{d} t \right \}.
\label{temp14}
\end{equation}
Without loss of generality, we apply a scaling transformation to the integral in equation $\eqref{temp14}$. Thus
\[
\int_0^{+\infty} W_n(\sigma-it) e^{i\omega t} \mathrm{d} t
= 
\frac{1}{\omega } \int_0^{+\infty} W_n\left (\sigma-i\frac{t}{\omega}\right ) e^{i t} \mathrm{d} t .
\]
Let $f_n(t,\omega) = -\frac{i}{\omega} W_n(\sigma - \frac{it}{\omega}) $, then
\begin{equation}
f_n(t,\omega) 
=  \prod_{v=0}^{n-1} \frac{t+i\omega(\sigma+\bar{\lambda}_v+1)}{t + i\omega (\sigma - \lambda_v)} \frac{1}{t+i\omega(\sigma-\lambda_{n})} . 
\label{fn_tw}
\end{equation}
Hence
% ?OE?$L_n(x)$??¡¥?¡·£€??????
\begin{equation}
    L_n(x) = \frac{x^\sigma}{\pi} \operatorname{Im} \left \{
\int_0^{+\infty} f_n(t,\omega)e^{it} \mathrm{d} t
\right \},
\label{numerical_integral1}
\end{equation}
where $\operatorname{Im}\{u\}$ denotes the image part of $u$.

\begin{theorem}
[\cite{milovanovic1999muntz}]
For any $ a > 0$,  the integral of $f_n(t,\omega)$can be transformed into
\begin{equation}
\int_{0}^{+\infty} f_n(t,\omega) e^{it}\mathrm{d} t 
 =  \int_{0}^a f_n(t,\omega)e^{it} \mathrm{d} t
+ 
ie^{ia} \int_0^{+\infty} f_n(a+iy,\omega) e^{-y} \mathrm{d} y.
\label{numerical_integral2}
\end{equation}
\end{theorem}

Thus, we have 
\begin{equation}
    L_n(x) = \frac{x^\sigma}{\pi} \operatorname{Im}\left \{ I^n_1 + I^n_2 \right \},
\end{equation}
where
\[
I^n_1 = \int_0^a  f_n(t,\omega) e^{it} \mathrm{d} t,\quad 
I^n_2 = i e^{i a} \int_0^{+\infty} f_n(a+i y,\omega) e^{-y} \mathrm{d} y.
\]
Therefore, it is need to accurately compute two parts: $x^\sigma$ and $I^n_1+I^n_2$. By a simple scaling argument, it follows that
\begin{equation}
\left |f_n(t,\omega) \right |  \leq 
\prod_{v=0}^{n-1} \frac{|\sigma + \bar{\lambda}_v+1|}{|\sigma-\lambda_v|}  \frac{1}{\omega|\sigma-\lambda_{n}|} .
\label{temp36}
\end{equation}
When $\sigma$ approaches $\lambda_v$ for some $v$ or $\omega$ approaches $0$, $f_n(t,\omega)$ exhibits singularity at $0$.  We aim to compute $L_n$ with the objective of ensuring that:
\begin{enumerate}
\item $f_n(t,w)$ does not exhibit singularity at $0$, so that $I^n_1+I^n_2$ can be accurately calculated numerically.
\item $x^\sigma$ does not amplify the error of $I^n_1+I^n_2$.
\end{enumerate}

Let $\sigma = \lambda_{min}- \theta / \omega$, where $\theta>0$ is a parameter to be determined, and $\lambda_{min} = \min_{v} {\lambda_v }$. Then we have
\[
\prod_{v=0}^{n-1} \frac{|\sigma+\bar{\lambda}_v+1|}{|\sigma-\lambda_v|}  \frac{1}{\omega|\sigma-\lambda_n|}
=
\prod_{v=0}^{n-1} \frac{|\theta-\omega(\lambda_{min} + \bar{\lambda}_v+1)|}{|\theta+\omega( \lambda_{min} + \lambda_v  )|}  \frac{1}{|\theta + \omega (\lambda_{min} + \lambda_n)|},
\]
and
\[x^\sigma = x^{\lambda_{min}} e^{\theta}.\]
The choice of $\sigma$ effectively avoids the singularity of $f_n(t,\omega)$ as $\omega$ approaches $0$.
However, when $\lambda_{min} + \lambda_v$ approaches $0$ for some $v$, we need to carefully choose the value of $\theta$. On the one hand, we cannot choose too large $\theta$ as it will result in large $x^\sigma$ and may cause numerical instability. On the other hand, we cannot choose a too small $\theta$ as it will make $f_n(t,\omega)$ singular near $0$, which also leads to numerical issues. Hence, it is crucial to strike a balance between these two considerations.
Let
\[
\mathcal{R}_n (\theta)= e^{\sqrt{\omega}} \prod_{v=0}^{n-1} \frac{|\theta-\omega(\lambda_{min} + \bar{\lambda}_v+1)|}{|\theta+\omega( \lambda_{min} + \lambda_v  )|}  \frac{1}{|\theta + \omega (\lambda_{min} + \lambda_n)|}
+ \frac{e^{\theta-\lambda_{min}\omega}}{\sqrt{\theta}}.
\]
We use the Nelder-Mead simplex algorithm \cite{lagarias1998convergence} to calculate $\theta = \mathrm{arg} \min_{\theta>0} \mathcal{R}_n (\theta)$.

For the calculation of $I^n_2$, we choose an appropriate $a>0$ and use Gauss-Laguerre quadrature to approximate it.

For the calculation of $I^n_1$, note that $f_n(t,\omega) e^{it}$ oscillates in $(0,a)$. When $x \in [\epsilon,1]$ for some small $\epsilon>0$, the oscillation factor $\omega$ falls within the range of  $[0,-\log \epsilon]$. Consequently, in double precision calculations, its value remains relatively small. Therefore, we use a piecewise approximation to calculate $I^n_1$. Let $a=mh$, where $h$ is a small constant and $m\in \mathbb{N}^+$ is a positive integer. Then
\begin{equation}
\label{piecewise_approximation}
I^n_1
= h \int_0^1 \sum_{k=1}^m f_n (h(y+k-1),\omega) e^{i h(y+k-1)} \mathrm{d} y
\end{equation}
can be approximated by using Gauss-Legendre quadrature. We denote $Q_1^n$ and $Q_2^n$  as the numerical approximations of $I_1^n$ and $I_2^n$, respectively.

In terms of implementation, the recurrence relation
\[
f_{n+1}(t,\omega) = f_n(t,\omega)  \frac{t+i\omega(\sigma + \bar{\lambda}_n + 1)}{t + i\omega (\sigma- \lambda_{n+1})}
\]
enables us to optimize the computation of all M\"{u}ntz-Legendre polynomials by dynamic programming. As a result, the computational cost of $L_0(x),L_1(x),\cdots,L_N(x)$ is comparable to that of $L_N(x)$. This process is summarized in the following \cref{alg:muntz_computing}.

\begin{algorithm}
\caption{Compute all M\"{u}ntz-Legendre polynomials for real M\"{u}ntz sequence}
\label{alg:muntz_computing}
\begin{algorithmic}
\REQUIRE{$\Lambda_N$ and $x$.}
\STATE{Choose $m$, $h$ in \cref{piecewise_approximation} and $M\in\mathbb{N}^+$.}
\STATE{Compute $\theta = \mathrm{arg} \min_{\theta>0} \mathcal{R}_N(\theta)$, $\omega=-\log x$ and $\sigma = \lambda_{min} - \theta/\omega$.}
\STATE{Compute Gauss-Legendre rule $\{\xi_j,\psi_j\}_{j=0}^M$
and Gauss-Laguerre rule $\{\tau_j,\eta_j\}_{j=0}^M$.}

\STATE{Compute $H_{j,k}^0 = f_0(h(\xi_j+k-1),\omega) e^{ih(\xi_j+k-1)}$ and 
$f_0(mh+i\tau_j,\omega)$, $j=0,1,\cdots,M$.}

\STATE{Let $L_0(x) \leftarrow x^{\lambda_0}$.}
\FOR{$n=1,2,\cdots,N$}
\FOR{$j=1,2,\cdots,M$}
\STATE{$$H_{j,k}^n = \frac{h(\xi_j+k-1) + i\omega(\sigma+\lambda_{n-1}+1)}{h(\xi_j+k-1) + i\omega(\sigma - \lambda_{n})} H_{j,k}^{n-1},\quad k=0,1,\cdots,m.$$}
\STATE{$$f_n(mh+i\tau_j,\omega) = \frac{mh+i\tau_j + i\omega (\sigma+\lambda_{n-1}+1)}{mh+i\tau_j + i\omega (\sigma-\lambda_n)} f_{n-1}(mh+i\tau_j,\omega). $$}
\ENDFOR
\STATE{Compute $Q_1^n \leftarrow h\sum_{j=1}^M \sum_{k=0}^m H_{j,k}^n \psi_j$ and $Q_2^n \leftarrow ie^{imh} \sum_{j=1}^M f_n(mh+i\tau_j,\omega) \eta_j$.}
\STATE{Update $L_n(x) \leftarrow \frac{x^\sigma}{\pi} \operatorname{Im}\{Q_1^n + Q_2^n\}$.}
\ENDFOR
\RETURN{$L_0(x), L_1(x), \cdots, L_N(x)$.}
\end{algorithmic}
\end{algorithm}

\section{Numerical construction of quadrature rule}
\label{sec:quadrature}
In this section, we present a method for the numerical construction of the generalized Gauss quadrature that is weighted by power function $\omega(x) = x^\beta$ for M\"{u}ntz polynomials. Assume that the M\"{u}ntz sequence $\Lambda_{2N+1}$ satisfies $\eqref{lambda_n}$ and consists solely of real indices. Our goal is to determine Gaussian nodes $x_k$ and weights $\omega_k$, $k=0,1,\cdots,N$, such that 
\begin{equation}
\sum_{k=0}^N L_{n}^\beta(x_k;\Lambda_n) \omega_k = m_n, \quad n=0,1,\cdots,2N+1,
\label{muntz_gauss_quadrature}
\end{equation}
where the moments $m_n$ are defined as
$$m_n = \int_{0}^1 L_n^\beta(x;\Lambda_n) \omega(x) \mathrm{d}x .$$
The existence and uniqueness of the Gauss quadrature \cref{muntz_gauss_quadrature} are guaranteed by \cref{generalized_gauss_quadrature1} and \cref{generalized_gauss_quadrature2}, as the set $\{L_0^\beta,\cdots,L_{2N+1}^\beta \}$ constitutes a Chebyshev system.

We define $\mathbf{c} = [m_0,\cdots,m_N]^\mathrm{T}$, $\mathbf{d} = [m_{N+1},\cdots,m_{2N+1}]^\mathrm{T}$, $\mathbf{m} = [\mathbf{c}^\mathrm{T},\mathbf{d}^\mathrm{T}]^\mathrm{T}$, $\bm{\omega} = [\omega_0,\cdots,\omega_{N}]^\mathrm{T}$,
$\mathbf{x}^{-\frac{\beta}{2}} = [x_0^{-\frac{\beta}{2}},\cdots,x_N^{-\frac{\beta}{2}}]^\mathrm{T}$, and the matrices $\mathbf{U}$ and $\mathbf{V}$
\begin{equation} \mathbf{U} = 
\left[ 
\begin{matrix}
L_0(x_0;\Lambda_0+\beta/2), & \cdots, & L_0(x_N;\Lambda_0+\beta/2)\\
L_1(x_0;\Lambda_1+\beta/2), & \cdots, & L_1(x_N;\Lambda_1+\beta/2)\\
\vdots & & \vdots\\ 
L_N(x_0;\Lambda_N+\beta/2), & \cdots, & L_N(x_N;\Lambda_N+\beta/2)\\
\end{matrix}
\right ],
\label{matrix_u}
\end{equation}
\begin{equation}
\mathbf{V} = 
\left[ 
\begin{matrix}
L_{N+1}(x_0;\Lambda_{N+1}+\beta/2), & \cdots, & L_{N+1}(x_N;\Lambda_{N+1}+\beta/2)\\
L_{N+2}(x_0;\Lambda_{N+2}+\beta/2), & \cdots, & L_{N+2}(x_N;\Lambda_{N+2}+\beta/2)\\
\vdots & & \vdots \\
L_{2N+1}(x_0;\Lambda_{2N+1}+\beta/2), & \cdots, & L_{2N+1}(x_N;\Lambda_{2N+1}+\beta/2)\\
\end{matrix}
\right ].
\label{matrix_v}
\end{equation}
Then, the equation \cref{muntz_gauss_quadrature} can be written in matrix form as
\begin{equation}
    \begin{bmatrix}
    \mathbf{U}\\ \mathbf{V}
    \end{bmatrix}
    \mathrm{diag} (\mathbf{x}^{-\frac{\beta}{2}})
    \bm{\omega}
    = 
    \begin{bmatrix}
    \mathbf{c} \\ 
    \mathbf{d} 
    \end{bmatrix},
\end{equation}
where $\mathrm{diag}(\mathbf{x})$ is a diagonal matrix with $\mathbf{x}$ as its entries.
Let $\bm{\Psi} = [\mathbf{U}^\mathrm{T},\mathbf{V}^\mathrm{T}]^\mathrm{T}$ and $\mathbf{z} = [\mathbf{x}^{\mathrm{T}}, \bm{\omega}^{\mathrm{T}}]^{\mathrm{T}}$. We construct a mapping $\mathbf{F}: \mathbb{R}^{2N+2}\rightarrow \mathbb{R}^{2N+2}$ such that
\begin{equation}
\mathbf{F}(\mathbf{z}) = \bm{\Psi} \mathrm{diag}(\mathbf{x}^{-\frac{\beta}{2}}) \bm{\omega} - \mathbf{m}.
\label{map_F}
\end{equation}
Then, there exists a unique solution $\mathbf{z}^*$ of $\mathbf{F}(\mathbf{z})=0$, which corresponds to the Gaussian nodes $\mathbf{x}^*$ and weights $\bm{\omega}^*$.
For a given $\mathbf{z}$, computing $\mathbf{F}(\mathbf{z})$ requires both $\bm{\Psi}$ and $\mathbf{m}$. As described in \cref{sec:muntz_computing}, $\bm{\Psi}$ can be computed by \cref{alg:muntz_computing}, while $\mathbf{m}$ can be recursively computed using \cref{moments_recurrence}.

Consider the Jacobian of $\mathbf{F}$
\begin{equation}
\mathbf{J}(\mathbf{z})
=
\partial_{\mathbf{z}} \mathbf{F} = 
\tilde{\mathbf{J}}(\mathbf{z})
\begin{bmatrix}
\mathrm{diag}(\mathbf{x}^{-\frac{\beta}{2}-1}) & \\
& \mathrm{diag}(\mathbf{x}^{-\frac{\beta}{2}})
\end{bmatrix}
\begin{bmatrix}
\mathrm{diag}(\bm{\omega}) & \\
 & \mathrm{diag}(\mathbf{1}) 
\end{bmatrix},
\label{matrix_j}
\end{equation}
where
\begin{equation}
    \tilde{\mathbf{J}}(\mathbf{z})
    =
    \begin{bmatrix}
    -\frac{\beta}{2} \mathbf{U} + \mathbf{U}^\prime \mathrm{diag}(\mathbf{x}) & \mathbf{U} \\
    -\frac{\beta}{2} \mathbf{V} + \mathbf{V}^\prime \mathrm{diag}(\mathbf{x}) & \mathbf{V}
    \end{bmatrix},
\end{equation}
\begin{equation} \mathbf{U}^\prime = 
\left[ 
\begin{matrix}
L^\prime_0(x_0;\Lambda_0+\beta/2), & \cdots, & L^\prime_0(x_N;\Lambda_0+\beta/2)\\
L^\prime_1(x_0;\Lambda_1+\beta/2), & \cdots, & L^\prime_1(x_N;\Lambda_1+\beta/2)\\
\vdots & & \vdots\\ 
L^\prime_N(x_0;\Lambda_N+\beta/2), & \cdots, & L^\prime_N(x_N;\Lambda_N+\beta/2)\\
\end{matrix}
\right ],
\label{matrix_y}
\end{equation}
and
\begin{equation}
\mathbf{V}^\prime = 
\left[ 
\begin{matrix}
L^\prime_{N+1}(x_0;\Lambda_{N+1}+\beta/2), & \cdots, & L^\prime_{N+1}(x_N;\Lambda_{N+1}+\beta/2)\\
L^\prime_{N+2}(x_0;\Lambda_{N+2}+\beta/2), & \cdots, & L^\prime_{N+2}(x_N;\Lambda_{N+2}+\beta/2)\\
\vdots & & \vdots \\
L^\prime_{2N+1}(x_0;\Lambda_{2N+1}+\beta/2), & \cdots, & L^\prime_{2N+1}(x_N;\Lambda_{2N+1}+\beta/2)\\
\end{matrix}
\right ].
\label{matrix_z}
\end{equation}

Clearly, the Jacobian matrix $\mathbf{J}(\mathbf{z})$ is a continuous mapping, and it is also nonsingular, as demonstrated in \cref{jacobi_nonsingular}. A zero of the function $\mathbf{F}$ where the Jacobian is nonsingular is referred to as a non-degenerate zero.

\begin{theorem}
\label{jacobi_nonsingular}
Let M\"{u}ntz sequence $\Lambda_{2N+1}$ consist solely of real indices and $\beta\in \mathbb{R}$ satisfy $\eqref{lambda_n}$. For any distinct $x_k \in (0,1)$ and $\omega_k > 0$, $k=0,1,\cdots,N$, $\mathbf{J}(\mathbf{z})$ is nonsingular.
\end{theorem}

The proof of \cref{jacobi_nonsingular} requires \cref{lem:zeros}.

\begin{lemma}
\label{lem:zeros}
For any distinct real numbers $z_0,z_1,\cdots,z_n$, we define
\[w_{n}(y) = \sum_{j=0}^n \left( h_j e^{z_j y} + l_j y e^{z_j(y-1)}  \right ),\quad y\in (-\infty,+\infty). \]
Thus, $w_n(y)$ contains at most $2n+1$ zeros if $\sum_{j=0}^n (h_j^2+l_j^2) > 0$.
\end{lemma}
\begin{proof}
By induction, when $n=0$, it is obvious that $h_0 e^{z_0 y} + l_0 y e^{z_0 (y-1)}$ contains at most $1$ zero if $h_0$ and $l_0$ are not all equal to zero. Assume that $w_{n-1}(y)$ contains at most $2n-1$ zeros if $\sum_{j=0}^{n-1} (h_j^2+l_j^2) > 0$. 

We claim that $w_{n}(y)$ contains at most $2n$ zeros if $l_n=0$ and $\sum_{j=0}^n ( h_j^2 +  l_j^2) > 0$. Let $F(y) = w_n(y) \exp (-z_n y)$, then 
\[ 
F^\prime(y) = \sum_{j=0}^{n-1} \left[\left(l_j e^{-z_j} + (z_j-z_n)h_j \right )e^{(z_j-z_n)y} + (z_j-z_n)l_j e^{-z_n} y e^{(z_j-z_n)(y-1)} \right ].
\]
By the induction hypothesis, $F^\prime (y)$ contains at most $2n-1$ zeros. Applying the inverse proposition of Rolle's theorem \cite{rudin1976principles}, both $F(y)$ and $w_n(y)$ contain at most $2n$ zeros.

When $l_n\neq0$ and $\sum_{j=0}^n (h_j^2+l_j^2)>0$, 
\[
F^\prime (y) = \sum_{j=0}^n \left[ (l_je^{-z_n} + h_j(z_j-z_n)) e^{(z_j-z_n)y} + l_j(z_j-z_n) e^{-z_n} y e^{(z_j-z_n)(y-1)} \right ] .
\]
Let $\tilde{h}_j = (l_je^{-z_n} + h_j(z_j-z_n))$ and $\tilde{l}_j = l_j(z_j-z_n) e^{-z_n}$. It can be observed that $\tilde{l}_n=0$ and $\sum_{j=0}^n ( \tilde{h}_j^2 + \tilde{l}_j^2)>0$. By the previously established claim, $F^\prime(y)$ contains at most $2n$ zeros, which implies that both $F(y)$ and $w_n(y)$ contain at most $2n+1$ zeros.
\end{proof}

\begin{proof}[Proof of \cref{jacobi_nonsingular}]
Without loss of generality, it suffices to consider the case when $\beta=0$. 
We begin our proof by assuming that all the elements in $\Lambda_{2N+1}$ are distinct. 
Let $\varphi_j(x) = x^{\lambda_j}$ for $j=0,1,\cdots,2N+1$.  It suffices to prove that the determinant $\operatorname{det} (\mathbf{Q})$ is non-zero, where
\begin{equation}
\mathbf{Q} = \left[
\begin{matrix}
\varphi_0\left(x_0\right) & \varphi_0^{\prime}\left(x_0\right) &  \cdots & \varphi_0\left(x_{n-1}\right) & \varphi_0^{\prime}\left(x_{n-1}\right) \\
\varphi_1\left(x_0\right) & \varphi_1^{\prime}\left(x_0\right) &  \cdots & \varphi_1\left(x_{n-1}\right) & \varphi_1^{\prime}\left(x_{n-1}\right) \\
\vdots & \vdots &  & \vdots & \vdots \\
\varphi_{2 n-1}\left(x_0\right) & \varphi_{2 n-1}^{\prime}\left(x_0\right) &  \cdots & \varphi_{2 n-1}\left(x_{n-1}\right) & \varphi_{2 n-1}^{\prime}\left(x_{n-1}\right)
\end{matrix}
\right].
\label{hermite_system}
\end{equation}

By \cref{lem:zeros}, for any distinct real numbers $z_0,z_1,\cdots,z_N$ and $y_0,y_1,\cdots,y_N$, the determinant of the matrix
\[
\left [
\begin{matrix}
e^{z_0 y_0} & y_0 e^{z_0 (y_0-1)} & \cdots & e^{z_N y_0} & y_0 e^{z_N(y_0-1)} \\
e^{z_0 y_1} & y_1 e^{z_0 (y_1-1)} & \cdots & e^{z_N y_1}  &  y_1 e^{z_N(y_1-1)} \\
\vdots & \vdots &  & \vdots & \vdots \\
e^{z_0 y_{2N+1}} & y_{2N+1} e^{z_0 (y_{2N+1}-1)} & \cdots & e^{z_N y_{2N+1}} & y_{2N+1} e^{z_N (y_{2N+1}-1)} 
\end{matrix}
\right ]
\]
does not vanish. By taking $z_j = \log x_j$ and $y_j = \lambda_j$, we  conclude that $\operatorname{det} (\mathbf{Q})\neq 0$.

When there are repeated indices in $\Lambda_{2N+1}$, the conclusion still holds due to the properties of determinant \cite{trefethen2022numerical} and \cref{repeated_indices}.
\end{proof}

To compute the matrix $\tilde{\mathbf{J}}$,  it is necessary to first calculate the matrix $\mathbf{U}$ and $\mathbf{V}$ using \cref{alg:muntz_computing}. These matrices are then combined with the recurrence relation
\begin{equation}
\begin{aligned}
xL_n^\prime (x;\Lambda_n+\beta/2) = xL_{n-1}^\prime (x;\Lambda_n& +\beta/2)  +
\left (\lambda_n+\frac{\beta}{2} \right ) L_n(x;\Lambda+\beta/2) \\ & + \left (1+\lambda_{n-1}  + \frac{\beta}{2} \right ) L_{n-1}(x;\Lambda_{n-1}+\beta/2),
\end{aligned}
\label{muntz_jacobi_recurrence}
\end{equation}
which can be directly obtained from \cref{muntz_recursive1}, to obtain both  $\mathbf{U}^\prime \mathrm{diag}(\mathbf{x})$ and $\mathbf{V}^\prime \mathrm{diag}(\mathbf{x})$ in a stable and accurate manner.
This process enables an efficient and reliable numerical computaition of $\tilde{\mathbf{J}}$.

We employ Newton's method to determine the zero of $\mathbf{F}$. Let $\mathbf{z}^{(0)}$ denote the initial guess, and $\mathbf{z}^{(k)}$ represent the $k$th iteration point. 
At the $k$th iteration, it is essential to solve the Newton equation
\begin{equation}
\mathbf{J}(\mathbf{z}^{(k)}) \mathbf{p}^{(k)} = -\mathbf{F}(\mathbf{z}^{(k)}).
\label{newton_equation}
\end{equation}
We denote $\tilde{\mathbf{p}}^{(k)}$ as
\begin{equation}
\tilde{\mathbf{p}}^{(k)} = 
\begin{bmatrix}
\mathrm{diag}(\mathbf{x}^{-\frac{\beta}{2}-1}) & \\
& \mathrm{diag}(\mathbf{x}^{-\frac{\beta}{2}})
\end{bmatrix}
\begin{bmatrix}
\mathrm{diag}(\bm{\omega}) & \\
 & \mathrm{diag}(\mathbf{1}) 
\end{bmatrix}
\mathbf{p}^{(k)} .
\end{equation}
Then $\tilde{\mathbf{J}}(\mathbf{z}^{(k)}) \tilde{\mathbf{p}}^{(k)} = \mathbf{J}(\mathbf{z}^{(k)}) \mathbf{p}^{(k)}$.
The Newton equation is transformed into
\begin{equation}
\tilde{\mathbf{J}}(\mathbf{z}^{(k)}) \tilde{\mathbf{p}}^{(k)} = -\mathbf{F}(\mathbf{z}^{(k)}).
\label{modified_newton_equation}
\end{equation}

In numerical computations, we solve \cref{modified_newton_equation} to obtain $\tilde{\mathbf{p}}^{(k)}$ instead of $\mathbf{p}^{(k)}$ as in \cref{newton_equation}. The reason for this modification is that when $x$ is close to 0 --- commonly encountered when $\hat{M}(\Lambda_{2N+1})$ consists of endpoint singular functions --- the computation of $\mathbf{J}$ can result in a loss of significant digits, which can cause the Newton iteration to converge slowly or even diverge.

Let $\tilde{\mathbf{p}}^{(k)}_1$ and $\mathbf{p}^{(k)}_1$ denote the first $N+1$ elements of $\tilde{\mathbf{p}}^{(k)}$ and $\mathbf{p}^{(k)}$ respectively.
% and $\tilde{\mathbf{p}}^{(k)}_2$ and $\mathbf{p}^{(k)}_2$ their last $N+1$ elements respectively.
Similarly, let $\tilde{\mathbf{p}}^{(k)}_2$ and $\mathbf{p}^{(k)}_2$ denote their last $N+1$ elements, respectively.
 Thus, we have
\begin{equation}
 \mathbf{p}^{(k)}_1 = \mathrm{diag}(\mathbf{x}^{(k)})^{\frac{\beta}{2}+1} \mathrm{diag}(\bm{\omega}^{(k)})^{-1} \tilde{\mathbf{p}}^{(k)}_1,\quad \mathbf{p}^{(k)}_2 = \mathrm{diag}(\mathbf{x}^{(k)})^{\frac{\beta}{2}} \tilde{\mathbf{p}}^{(k)}_2.
\end{equation}
From the definition of $\mathbf{z}$, one can see that $\mathbf{p}^{(k)}_1$ and $\mathbf{p}^{(k)}_2$ correspond to the Newton descent directions of $\mathbf{x}^{(k)}$ and $\bm{\omega}^{(k)}$, respectively. Update the Gaussian nodes and weights at $k$th iteration by
\begin{equation}
\bm{\omega}^{(k+1)} \leftarrow \bm{\omega}^{(k)} + \mathbf{p}^{(k)}_2,\quad
\mathbf{x}^{(k+1)} \leftarrow \mathbf{x}^{(k)} + \mathbf{p}^{(k)}_1. 
\label{newton_update}
\end{equation}

We can prove that our method for this modification is convergent locally. The classic theorem on the Newton's method can be summarized as follows.
\begin{theorem}
[\cite{jorge2006}]
\label{newton_convergence}
Let $\mathbf{F}:\mathbb{R}^m \rightarrow \mathbb{R}^m$ be a differentiable mapping with a non-degenerate zero $\mathbf{z}^*$ such that $\mathbf{F}(\mathbf{z}^*) = \mathbf{0}$ and $\mathrm{det} \left(\mathbf{J}(\mathbf{z}^*) \right) \neq 0$, where $\mathbf{J}(\mathbf{z})$ denotes the Jacobian matrix of $\mathbf{F}$ evaluated at $\mathbf{z}$. Then, there exists $\varepsilon>0$ such that for any initial guess $\mathbf{z}^{(0)}$ satisfying $\| \mathbf{z}^{(0)} - \mathbf{z}^* \| < \varepsilon $, the iterates generated by the update rule
\[
\mathbf{z}^{(k+1)} = \mathbf{z}^{(k)} - \mathbf{J}(\mathbf{z}^{(k)})^{-1} \mathbf{F}(\mathbf{z}^{(k)}),\quad k=0,1,2,\cdots.
\]
satisfy $\mathbf{z}^{(k+1)}-\mathbf{z}^* = o(\|\mathbf{z}^{(k)}-\mathbf{z}^*\|)$. Moreover, if $\mathbf{J}(\mathbf{z})$ is Lipschitz continuous in a neighborhood of $\mathbf{z}^*$ with radius $\varepsilon$, then $\mathbf{z}^{(k+1)}-\mathbf{z}^* = O(\|\mathbf{z}^{(k)}-\mathbf{z}^*\|^2)$.
\end{theorem}

As all the Gaussian nodes lie in the interior of the integration interval and all the Gaussian weights are positive,  it follows that there exists some $\varepsilon > 0$ such that the Jacobian $\mathbf{J}(\mathbf{z})$ is Lipschitz continuous in a neighborhood of $\mathbf{z}^*$ with radius $\varepsilon$. Thus the following result holds.

\begin{lemma}
Provided that the starting values are sufficiently good, the presented
form of the Newton's method \cref{newton_update} is convergent quadratically.
\end{lemma}

In practical implementation, Newton's method usually converges after only a few iterations. However, in cases where slow convergence occurs, we use a damped Newton's method to dampen the step size as $s^{(k)} = \gamma^{\max(0,k-k_0)}$, where $0<\gamma < 1$ and $k_0 \in \mathbb{N}$.

%\subsection{Continuation method} 
Additionally, in terms of the local convergence property of the Newton's method, it is imperative to select an optimal initial guess. However, obtaining a suitable initial point can be a challenging endeavor in various scenarios. As a result, we resort to the Continuation method as a means of obtaining the initial guess.

Let $\alpha \in [0,1]$.
%and $\Lambda_{2N+1}$ defined as in $\eqref{lambda_n}$. 
Set
$$
\Lambda_{2N+1}^{(\alpha)} = \left \{\lambda_0^{(\alpha)}, \lambda_1^{(\alpha)},\cdots,\lambda_{2N+1}^{(\alpha)}\right \},
$$
where
\begin{equation}\lambda_n^{(\alpha)} = \alpha \lambda_n + (1-\alpha)n,\quad 
n = 0,1,\cdots, 2N+1.
\label{temp41} \end{equation}
Then $L_n^\beta(x;\Lambda_n^{(\alpha)})\in \hat{M}\left(\Lambda_{2N+1}^{(\alpha)}\right )$ are the orthogonal M\"{u}ntz polynomials with respect to weight function $\omega(x)$. Again, by applying the \cref{generalized_gauss_quadrature1} and \cref{generalized_gauss_quadrature2},
there exists the unique Gaussian nodes $x_k(\alpha)\in (0,1)$ and weights $\omega_k(\alpha)>0$, $k=0,1,\cdots,N$, such that
\begin{equation}
\sum_{k=0}^N L_n^\beta (x_k(\alpha);\Lambda_n^{(\alpha)}) \omega_k(\alpha) = m_n(\alpha),\quad n=0,1,\cdots,2N+1,
\label{temp70}
\end{equation}
where
$m_n(\alpha) = \int_0^1 L_n^\beta ( x;\Lambda_n^{(\alpha)}) \omega(x) \mathrm{d}x$. Using the Newton's method outlined above with an appropriate initial guess, we are able to compute Gaussian nodes $x_k(\alpha)$ and weights $\omega_k(\alpha)$, moreover, both of which are continuous with respect to $\alpha$.
\begin{theorem} \label{continuity}
The M\"{u}ntz sequence $\Lambda_{2N+1}$ and $\beta\in\mathbb{R}$ are defined in $\eqref{lambda_n}$. For any $\alpha \in [0,1]$, we define $\Lambda_{2N+1}^{(\alpha)}$ according to $\eqref{temp41}$. The Gaussian nodes $x_k(\alpha)$ and weights $\omega_k(\alpha)$ satisfy $\eqref{temp70}$ and exhibit continuity with respect to $\alpha$.
\end{theorem}
\begin{proof}
Let 
$\mathbf{z}(\alpha) = [x_0{(\alpha)},\cdots,x_N{(\alpha)},\omega_0(\alpha),\cdots,\omega_N(\alpha)]^\mathrm{T}
$. Our goal is to show that $\mathbf{z}(\alpha)$, or its every component, is continuous.
Based on $\eqref{map_F}$, we can similarly establish a mapping $\mathbf{F}(\mathbf{z},\alpha)$ associated with the M\"{u}ntz sequence $\Lambda_{2N+1}^{(\alpha)}$ such that
\begin{equation}
    \mathbf{F}(\mathbf{z}(\alpha),\alpha) = \mathbf{0}.
    \label{temp71}
\end{equation}
As $L_n^{\beta}(x;\Lambda_n^{(\alpha)})$ is differentiable with respect to $\alpha$, it follows that $\mathbf{F}(\mathbf{z},\alpha)$ is differentiable with respect to $\alpha$, i.e., $\partial_\alpha \mathbf{F}$ exists.
Taking the derivative of $\eqref{temp71}$ with respect to $\alpha$, we obtain
\[\frac{\partial \mathbf{F}}{\partial \alpha} + \frac{\partial \mathbf{F}}{\partial \mathbf{z}} \frac{\partial \mathbf{z}}{\partial \alpha}  = \mathbf{0},
\]
where $\partial_\alpha \mathbf{F}$ and $\partial_\alpha \mathbf{z}$ are $2N+2$ dimensional column vectors. ${\partial_{\mathbf{z}} \mathbf{F}}$ is a $2N+2$ order matrix, which is the Jacobian matrix in $\eqref{matrix_j}$. Therefore, $\partial_\mathbf{z} \mathbf{F}$ is nonsingular, implying that $\partial_\alpha \mathbf{z}$ exists and hence $\mathbf{z}(\alpha)$ is continuous.
\end{proof}
\begin{remark}
If $\alpha=0$, Gaussian nodes $\mathbf{x}(\alpha)$ and weights $\bm{\omega}(\alpha)$ reduce to the classic Gauss-Jacobi nodes and weights, which can be obtained by various well-known methods in \cite{golub1969calculation,glaser2007fast}. When $\alpha$ equals $1$, $\mathbf{x}(\alpha)$ and $\bm{\omega}(\alpha)$ are the desired ones.
\end{remark}

To obtain the Gaussian nodes and weights corresponding to a specific M\"{u}ntz sequence, we first compute the initial iterate $\mathbf{z}(0)$. We then incrementally increase $\alpha$ by a small step size $\Delta \alpha$ and use Newton's method with the previous iterate $\mathbf{z}(\alpha-\Delta \alpha)$ as the initial guess to compute $\mathbf{z}(\alpha)$. By repeating this process with sufficiently small $\Delta \alpha$, the continuity of $\mathbf{z}(\alpha)$ ensures that the desired Gaussian nodes and weights can be obtained when $\alpha=1$. This process is summarized in the following \cref{alg:gauss_quadrature}.
%However, it is important to note that the choice of step size may impact the accuracy and computational efficiency of the method, and further analysis may be necessary to determine an appropriate range for $\Delta \alpha$.

\begin{algorithm}
\caption{Compute Gauss quadrature}
\label{alg:gauss_quadrature}
\begin{algorithmic}
\REQUIRE{$\Lambda_{2N+1}$ and $\beta\in \mathbb{R}$ satisfy \eqref{lambda_n}.}
\STATE{Compute $\mathbf{z}(0)$ and set $\alpha\leftarrow \Delta \alpha$.}
\WHILE{$\alpha \leq 1 $}
\STATE{Compute moments $\mathbf{m}(\alpha)$ by \cref{moments_recurrence}.}
\STATE{Determine the initial guess in Newton's method $\mathbf{z}^{(0)} \leftarrow \mathbf{z}(\alpha-\Delta \alpha)$.}
\FOR{$k=1,2,\cdots$}
\STATE{Compute $\tilde{\mathbf{J}}(\mathbf{z}^{(k)})$ and $\mathbf{F}(\mathbf{z}^{(k)})$ by \cref{alg:muntz_computing} and $\cref{muntz_jacobi_recurrence}$.}
\STATE{Compute the descent directions $\tilde{\mathbf{p}}^{(k)}_1$ and $\tilde{\mathbf{p}}^{(k)}_1$ by $\eqref{modified_newton_equation}$.}
\STATE{Update weights: $\bm{\omega}^{(k+1)} \leftarrow \bm{\omega}^{(k)} + s^{(k)} 
 \mathrm{diag}(\mathbf{x}^{(k)})^{\frac{\beta}{2}} \tilde{\mathbf{p}}^{(k)}_2$.}
\STATE{Update nodes: $\mathbf{x}^{(k+1)} \leftarrow \mathbf{x}^{(k)} + s^{(k)} \mathrm{diag}(\mathbf{x}^{(k)})^{\frac{\beta}{2}+1} \mathrm{diag}(\bm{\omega}^{(k)})^{-1} \tilde{\mathbf{p}}^{(k)}_1$.}
\STATE{$\mathbf{z}^* \leftarrow (\mathbf{x}^{(k+1)}, \bm{\omega}^{(k+1)})$.}
\ENDFOR
\STATE{Obtain $\mathbf{z}(\alpha) \leftarrow \mathbf{z}^*$.}
\STATE{Update $\alpha \leftarrow \alpha + \Delta \alpha$.}
\ENDWHILE
\RETURN $\mathbf{z}(1)$.
\end{algorithmic}
\end{algorithm}

\begin{remark}
In \cref{alg:gauss_quadrature}, an unordered M\"{u}ntz sequence is allowed. Moreover, the continuation step size $\Delta \alpha$ is not necessarily the same at each iteration.
\end{remark}

Specifically, Gauss quadrature with $\omega(x) = 1$ can be obtained via a transformation of that with $\omega(x) = x^\beta$.

\begin{theorem}\label{temp47}
Let $\{x_j,\omega_j \}_{j=0}^N$ be the Gaussian nodes and weights with weight function $\omega(x) = x^{\beta}$ that is exact for functions in $\hat{M}(\Lambda_{2N+1})$. Let $\kappa = 1/(\beta+1)$, 
\[
\tau_{j} = x_j^{\frac{1}{\kappa}}, \quad \chi_{j} = \frac{\omega_j}{\kappa}, \quad j=0,1,\cdots,N.
\]
Then $\{\tau_j,\chi_j\}_{j=0}^N$ are the Gaussian nodes and weights with weight function $\omega(x)=1$ over $\hat{M}\left( \kappa \Lambda_{2N+1} \right)$, i.e.,
\begin{equation}
\int_0^1 p(x) \mathrm{d} x =
\sum_{j=0}^N {p(\tau_j)} \chi_j,\quad \forall p(x)\in \hat{M}\left( \kappa \Lambda_{2N+1} \right).
\end{equation}
\end{theorem}
\begin{proof}
For any $p(x) \in \hat{M}(\kappa \Lambda_{2N+1})$, consider the variable transformation $x = t^{\frac{1}{\kappa}}$. Then, $p(t^{\frac{1}{\kappa}})\in \hat{M}(\Lambda_{2N+1})$, and as a result, 
\[
\begin{aligned}
\int_0^1 p(x) \mathrm{d} x = \frac{1}{\kappa}\int_0^1 p(t^{\frac{1}{\kappa}}) t^{\frac{1}{\kappa}-1} \mathrm{d} t
= \frac{1}{\kappa}\int_0^1 p(t^{\frac{1}{\kappa}}) t^{\beta} \mathrm{d} t
=\frac{1}{\kappa}\sum_{j=0}^N p(x_j^{\frac{1}{\kappa}}) \omega_j
= \sum_{j=0}^N p(\tau_j) \chi_j.
\end{aligned}\]
\end{proof}

\section{Error estimation}
\label{sec:error}
This section discusses the error estimation for a specific type of M\"{u}ntz Gauss quadrature and shows that its convergence rate is independent of the integrand's singularity.

For interpolation type numerical quadrature, error estimation is usually given in terms of the derivative of the integrand, and the order of the derivative depends on the number of quadrature nodes. Such estimates can be obtained from polynomial interpolation error estimates or the Peano kernel theorem. Estimating the interpolation error for general functions that form a Chebyshev system is difficult. For special M\"{u}ntz polynomials, however, we can use the Peano kernel theorem to obtain an error estimate for the quadrature.

Let $\mathcal{V}[a, b]$ be the collection of all bounded variation real-valued functions
%\cite{royden1988real}
on the interval $[a,b]$. Suppose that $\mathscr{F}: \mathcal{V}[a, b] \rightarrow \mathbb{R}$ is a bounded linear functional 
%\cite{brezis2011functional}
, and $K(\theta)$ is the Peano kernel  \cite{powell1981approximation,huybrechs2009generalized} with respect to $\mathscr{F}$ and $n \in \mathbb{N}$:
\begin{equation}
K (\theta) = \frac{1}{n!} \mathscr{F}_x[(x-\theta)_+^n],\quad \theta \in [a,b],
\label{peano_kernal}
\end{equation}
where $\mathscr{F}_x[\cdot]$ denotes the action of the functional $\mathscr{F}$ on a function with respect to $x$, and 
\[
(x-\theta)_+^n = \left \{ 
\begin{aligned}
& (x-\theta)^n, & x\geq \theta, \\
&0, & x<\theta. \\
\end{aligned}
\right.
\]
Based on Taylor's theorem
%\cite{sauer2011numerical}
 and the expression for the remainder in Taylor's series, the Peano kernel theorem can be proved \cite{powell1981approximation}. The Peano kernel theorem provides a useful expression: the action of a linear functional on a function can be represented as the integral of the function's derivative and the Peano kernel.
\begin{theorem}[\cite{powell1981approximation,huybrechs2009generalized,gautschi2011numerical}]\label{peano_thm}
Let $n\in \mathbb{N}$ and $\mathscr{F}: \mathcal{V}[a, b] \rightarrow \mathbb{R}$ be a bounded linear functional satisfying $\mathscr{F}[p(x)]=0$ for all $p(x)\in\mathbb{P}_n$, where $\mathbb{P}_n$ represents the collection of all polynomials with degree less than or equal to $n$. $K(\theta)$ is defined as in \eqref{peano_kernal} and $K(\theta)\in \mathcal{V}[a, b]$. Then, for any $f(x) \in C^{n+1}[a, b]$, we have
\begin{equation}
\mathscr{F}[f] =
\int_0^1 K(\theta) f^{(n+1)}(\theta) \mathrm{d} \theta.
\label{peano_functional}
\end{equation}
\end{theorem}

From \cref{peano_functional}, it is evident that the following estimate holds:
\begin{equation}
\left |\mathscr{F}[f] \right | \leq \|K \|_1 \|f^{(n+1)} \|_{\infty}.
\label{peano_estimate}
\end{equation}
Therefore, if we choose a specific $\mathscr{F}$ that satisfies the conditions of \cref{peano_thm}, we can obtain an estimate for $|\mathscr{F}[f]|$ using \cref{peano_estimate}. In particular, we utilize it to estimate the error of a certain class of M\"{u}ntz Gauss quadratures.

\begin{lemma}
[\cite{huybrechs2009generalized}]
\label{estimation_lemma}
Let $\omega(x)$ be a weight function. The functional 
\begin{equation}
\mathscr{F}[u] = I[u] - Q_N[u] 
\end{equation}
 defines the error in the numerical approximation of the exact integral $I[u]$ as defined in \cref{symbol_integral}, using Gauss quadrature $Q_N[u]$ as defined in \cref{quadrature1}.
If $\mathscr{F}[p(x)] = 0$, $\forall p(x) \in \mathbb{P}_N$, then for any $u\in C^{N+1}[a,b]$, it follows that 
\begin{equation}
\left| I[u] - Q_N[u]\right | \leq \frac{(b-a)^{N+1}}{N!} \int_a^b \omega(x) \mathrm{d} x \|u^{(N+1)} \|_\infty.
\end{equation} 
\end{lemma}

We consider the interval $[0,1]$ and M\"{u}ntz sequence
\begin{equation}
\Lambda_{2N+1} 
= 
\left \{\lambda_v = \left \lfloor \frac{v}{2} \right \rfloor: v=0,1,\cdots,2N+1  \right \},
\label{temp42}
\end{equation}
where $\lfloor \cdot \rfloor$ denotes the floor function. Let $\omega(x) = x^\beta$ satisfy \cref{lambda_n} and $Q_N[\cdot]$ be the associated Gaussian rule. We have the following estimation.
%\begin{lemma}
%\label{peano_thm1}
%Let $\Lambda_{2N+1}$ be defined as in $\eqref{temp42}$.
%Then for any $u\in C^{N+1}[a,b]$, we have 
%\begin{equation}
%\left |I[u]-Q_N[u] \right | \leq 
%\frac{1}{N!}\cdot \frac{1}{1+\beta} \|u^{(N+1)} \|_\infty .
%\label{quadrature_error1}
%\end{equation}
%\end{lemma}

\begin{theorem}\label{peano_thm3}
Let $\Lambda_{2N+1}$ be defined as in $\eqref{temp42}$.
For any $u,v \in C^{N+1}[0,1]$, we set $f(x) = u(x) + v(x)\log x$, then
\begin{equation}
\left |I[f]-Q_N[f] \right| \leq 
\frac{1}{N!} \left( 
\frac{1}{1+\beta} \|u^{(N+1)} \|_\infty 
+
\frac{1}{(1+\beta)^2} \|v^{(N+1)} \|_\infty
\right ).
\label{quadrature_error3}
\end{equation}
\end{theorem}
\begin{proof}
We firstly take $\mathscr{F}[u] = I[u] - Q_N[u]$ and $\omega(x) = x^\beta$. Then we have $\mathscr{F}[p]=0$, $\forall p\in \mathbb{P}_N$. By \cref{estimation_lemma}, we have 
\begin{equation}
\left |I[u]-Q_N[u] \right| \leq \frac{1}{N!} \frac{1}{1+\beta} \|u^{(N+1)} \|_\infty.
\end{equation}
Secondly, we denote $\tilde{I}[v]$ and $\tilde{Q}_N[v]$ by
\[
\tilde{I}[v] = I[v(-\log x)],\quad \tilde{Q}_N [v] = Q_N[v(-\log x)].
\]
Let  $\mathscr{F}[v] = \tilde{I}[v] - \tilde{Q}_N [v]$. Then we have $\mathscr{F}[p]=0$, $\forall p\in \mathbb{P}_N$. Taking replace $\omega(x)$ with $\omega(x) (-\log x)$ in \cref{estimation_lemma}, we have 
\begin{equation}
\left |I[v(-\log x)]-Q_N[v(-\log x)] \right| \leq \frac{1}{N!} \frac{1}{(1+\beta)^2} \|v^{(N+1)} \|_\infty.
\end{equation}
Hence \cref{quadrature_error3} follows from combining 
\[
\left |I[v(-\log x)]-Q_N[v(-\log x)] \right| = \left |I[v\log x]-Q_N[v\log x] \right|
\]
with
\[
\left |I[f]-Q_N[f] \right| \leq \left |I[u]-Q_N[u] \right| + \left |I[v\log x]-Q_N[v\log x] \right|.
\]
%This complete the proof.
\end{proof}

According to \cref{peano_thm3}, the error in Gauss quadrature for singular functions at endpoints, of the form $u(x) + v(x)\log x$, is not influenced by $\log x$, but rather determined by the smoothness of $u(x)$ and $v(x)$.

Generally, let $T \in \mathbb{N}^+$, we consider M\"{u}ntz sequence as
\begin{equation}
\label{muntz_T}
\Lambda_{T(N+1)} =\left \{ \lambda_v =  \left \lfloor \frac{v}{T} \right \rfloor: v=0,1,\cdots,T(N+1) \right \}.
\end{equation}
Let $\omega(x) = x^\beta$ and $S = \left \lceil \frac{T(N+1)}{2} \right \rceil$. $\lceil \cdot \rceil$ denotes the ceil function. $Q_S[\cdot]$ are the associated Gauss quadrature. Similar to the result stated in \cref{peano_thm3}, the following estimation holds.
\begin{theorem}
\label{error_estimation}
Let $\Lambda_{T(N+1)}$ be defined as in \cref{muntz_T} and $\beta\in\mathbb{R}$ satisfy \cref{lambda_n}. For any $u_j \in C^{N+1}[0,1]$, we set $f(x) = \sum_{j=0}^{T-1} u_j(x) \log^j x$, then
\begin{equation}
\left | I[f] - Q_S[f]\right | \leq \frac{1}{N!} \sum_{j=0}^{T-1} \frac{\Gamma(j+1)}{(1+\beta)^{j+1}} \|u_j^{(N+1)} \|_\infty,
\end{equation}
where $\Gamma$ is the Gamma function and $S=\left \lceil \frac{T(N+1)}{2} \right \rceil$.
\end{theorem}

\section{Numerical examples}
\label{sec:examples}
In this section we present some numerical examples.
\begin{example}
\label{exm1}
Suppose that $\beta=-1/4$ and M\"{u}ntz sequence $\Lambda_{2N+1}$ satisfies
\[
\lambda_{2k} = k+\frac{2}{3}, \quad \lambda_{2k+1} = k-\frac{2}{3}, \quad k=0,1,\cdots,N.
\]
Gaussian nodes $x_k$ and weights $\omega_k$, $k=0,1,\cdots,N$, as shown in \cref{tab:quadrature1}, such that 
\[
\int_0^1 \varphi(x) \omega(x) \mathrm{d} x = \sum_{k=0}^N \varphi(x_k) \omega_k,\quad \forall \varphi(x)\in \hat{M}(\Lambda_{2N+1}).
\]
Note that $\Lambda_{2N+1}$ is not monotonous and its minimum index is $-{2}/{3}$. Therefore, some functions in $\hat{M}(\Lambda_{2N+1})$ may not even be in $L^2[0,1]$.
From \cref{tab:quadrature1}, it can be observed that the distribution of Gaussian nodes is more concentrated around 0, which is suitable for handling cases where the integrand grows or decreases rapidly near the left endpoint of the interval. In \cref{tab:err1}, we present a list of relative errors, denoted as 
\begin{equation}
R[\varphi] = \bigg |\frac{ I[\varphi] - Q_N[\varphi] }{ I[\varphi]} \bigg |,
\label{eq:relative_err}
\end{equation}
where $I[\varphi]$ represents the exact integral as defined in \cref{symbol_integral}, and $Q_N[\varphi]$ denotes the quadrature as given in \cref{quadrature1} with $N+1$ nodes and weights listed in \cref{tab:quadrature1}. 
The results in \cref{tab:err1} demonstrate that the $N+1$ quadrature rule $Q_N[\cdot]$ achieves exactness, within machine accuracy, for any $\varphi \in \hat{M}(\Lambda_{2N+1})$.
\end{example}

{\footnotesize
\begin{longtable}{|p{0.8cm}|p{3.4cm}|p{3.4cm}|}
\caption{\rm Gauss quadrature of \cref{exm1}.}
\label{tab:quadrature1}\\
\hline
$N+1$ & Nodes $x_k$  & Weights $\omega_k$  \\
\hline 
\endfirsthead
\multicolumn{3}{r}{\textit{Continued on next page}} \\
\endfoot
\endlastfoot
 \ \ 20& 2.3157766972828912(-6) & 9.4222433583541251(-4) \\ 
 & 2.7233174824378183(-4) & 5.2428252534242620(-3) \\ 
 & 1.8028430213927918(-3) & 1.2949690421473462(-2) \\ 
 & 6.2586172324495502(-3) & 2.3479387650276507(-2) \\ 
 & 1.5778280059386831(-2) & 3.6207147638052821(-2) \\ 
 & 3.2714449729529464(-2) & 5.0373872558283253(-2) \\ 
 & 5.9338006505081552(-2) & 6.5125430600342371(-2) \\ 
 & 9.7535823413401598(-2) & 7.9558300013059469(-2) \\ 
 & 1.4853234267777532(-1) & 9.2767975279634124(-2) \\ 
 & 2.1266338866366832(-1) & 1.0389748261138633(-1) \\ 
 & 2.8922591041325307(-1) & 1.1218341981654856(-1) \\ 
 & 3.7642038226828578(-1) & 1.1699708689403455(-1) \\ 
 & 4.7139401605805348(-1) & 1.1787854613502737(-1) \\ 
 & 5.7038355874684765(-1) & 1.1456183436550083(-1) \\ 
 & 6.6894714365725705(-1) & 1.0699002517468645(-1) \\ 
 & 7.6226630027383013(-1) & 9.5319381997936783(-2) \\ 
 & 8.4549259057549253(-1) & 7.9912437335814088(-2) \\ 
 & 9.1410908658850543(-1) & 6.1320533998553618(-2) \\ 
 & 9.6427580053292161(-1) & 4.0258021409587424(-2) \\ 
 & 9.9313650659281172(-1) & 1.7593770011811099(-2) \\ 
\hline
\hline 
 \ \ 40 & 1.5187265199530925(-7) & 1.2213355322124739(-4) \\ 
 & 1.7992560515276967(-5) & 6.8605007097713726(-4) \\ 
 & 1.2069214706402429(-4) & 1.7228923928849940(-3) \\ 
 & 4.2706310022074624(-4) & 3.2002598939306478(-3) \\ 
 & 1.1039676484141875(-3) & 5.0956596491180742(-3) \\ 
 & 2.3612508697757515(-3) & 7.3804450403362623(-3) \\ 
 & 4.4453368850422828(-3) & 1.0020032335832917(-2) \\ 
 & 7.6316415093306460(-3) & 1.2974358916941377(-2) \\ 
 & 1.2215994413850436(-2) & 1.6198429091116944(-2) \\ 
 & 1.8505288925579538(-2) & 1.9642930054004918(-2) \\ 
 & 2.6807597817191973(-2) & 2.3254907906113849(-2) \\ 
 & 3.7422007865292671(-2) & 2.6978494103994850(-2) \\ 
 & 5.0628434205556849(-2) & 3.0755672338654855(-2) \\ 
 & 6.6677677381083530(-2) & 3.4527075372775690(-2) \\ 
 & 8.5781981413664360(-2) & 3.8232800993517967(-2) \\ 
 & 1.0810634033123627(-1) & 4.1813235992838205(-2) \\ 
 & 1.3376078361624985(-1) & 4.5209876987019809(-2) \\ 
 & 1.6279384840367642(-1) & 4.8366136936514245(-2) \\ 
 & 1.9518741849553753(-1) & 5.1228126426811106(-2) \\ 
 & 2.3085307803670563(-1) & 5.3745399118535722(-2) \\ 
 & 2.6963009178752406(-1) & 5.5871651265625516(-2) \\ 
 & 3.1128508519477094(-1) & 5.7565365827494788(-2) \\ 
 & 3.5551345683402841(-1) & 5.8790392455852776(-2) \\ 
 & 4.0194251424969035(-1) & 5.9516455508932177(-2) \\ 
 & 4.5013628275188522(-1) & 5.9719583223261812(-2) \\ 
 & 4.9960189633891738(-1) & 5.9382452242441777(-2) \\ 
 & 5.4979744157007604(-1) & 5.8494642849110580(-2) \\ 
 & 6.0014108983850478(-1) & 5.7052801454937990(-2) \\ 
 & 6.5002132193804196(-1) & 5.5060708157983945(-2) \\ 
 & 6.9880802184137081(-1) & 5.2529248460949217(-2) \\ 
 & 7.4586419486102140(-1) & 4.9476289542301162(-2) \\ 
 & 7.9055804937950391(-1) & 4.5926462772571036(-2) \\ 
 & 8.3227517151117658(-1) & 4.1910855467128326(-2) \\ 
 & 8.7043051867701615(-1) & 3.7466616191765861(-2) \\ 
 & 9.0447996135064512(-1) & 3.2636479428631744(-2) \\ 
 & 9.3393111257234906(-1) & 2.7468217697410070(-2) \\ 
 & 9.5835320406140556(-1) & 2.2014035553052494(-2) \\ 
 & 9.7738580706529599(-1) & 1.6329952018268864(-2) \\ 
 & 9.9074634991357979(-1) & 1.0475501181546670(-2) \\ 
 & 9.9823822268456675(-1) & 4.5199615467215945(-3) \\ 
  \hline 
 \end{longtable}
 }

\begin{table}[htbp]
\centering \small
\caption{\rm Relative error of Gauss quadrature in \cref{exm1}.}
\label{tab:err1}  
\begin{tabular}{|c|c|c|c|c|}
\hline 
$N+1$ & $\varphi(x) $ & Relative error $R[\varphi]$ & $\varphi(x)$ & Relative error $R[\varphi]$ \\
\hline 
 20 & $x^{2/3}$ & 1.1102230246251565(-15) & $x^{32/3}$ & 1.8873791418627661(-15) \\ 
 & $x^{-2/3}$ & 7.7715611723760958(-16) & $x^{28/3}$ & 1.7763568394002505(-15) \\ 
 & $x^{5/3}$ & 1.5543122344752192(-15) & $x^{35/3}$  & 1.8873791418627661(-15) \\ 
 & $x^{1/3}$ & 2.2204460492503131(-16) & $x^{31/3}$ & 1.8873791418627661(-15) \\ 
 & $x^{8/3}$ & 1.6653345369377348(-15) & $x^{38/3}$ & 1.8873791418627661(-15) \\ 
 & $x^{4/3}$ & 1.2212453270876722(-15) & $x^{34/3}$ & 1.6653345369377348(-15) \\ 
 & $x^{11/3}$ & 1.5543122344752192(-15) & $x^{41/3}$ & 1.9984014443252818(-15) \\ 
 & $x^{7/3}$ & 1.5543122344752192(-15) & $x^{37/3}$ & 1.9984014443252818(-15) \\ 
 & $x^{14/3}$ & 1.6653345369377348(-15) & $x^{44/3}$ & 1.8873791418627661(-15) \\ 
 & $x^{10/3}$ & 1.3322676295501878(-15) & $x^{40/3}$ & 1.9984014443252818(-15) \\ 
 & $x^{17/3}$ & 1.7763568394002505(-15) & $x^{47/3}$ & 1.9984014443252818(-15) \\ 
 & $x^{13/3}$ & 1.8873791418627661(-15) & $x^{43/3}$ & 1.8873791418627661(-15) \\ 
 & $x^{20/3}$ & 1.6653345369377348(-15) & $x^{50/3}$ & 1.8873791418627661(-15) \\ 
 & $x^{16/3}$ & 1.8873791418627661(-15) & $x^{46/3}$ & 1.8873791418627661(-15) \\ 
 & $x^{23/3}$ & 1.6653345369377348(-15) & $x^{53/3}$ & 2.2204460492503131(-15) \\ 
 & $x^{19/3}$ & 1.5543122344752192(-15) & $x^{49/3}$ & 1.9984014443252818(-15) \\ 
 & $x^{26/3}$ & 1.9984014443252818(-15) & $x^{56/3}$ & 2.1094237467877974(-15) \\ 
 & $x^{22/3}$ & 1.7763568394002505(-15) & $x^{52/3}$ & 1.9984014443252818(-15) \\ 
 & $x^{29/3}$ & 1.7763568394002505(-15) & $x^{59/3}$ & 2.1094237467877974(-15) \\ 
 & $x^{25/3}$ & 1.6653345369377348(-15) & $x^{55/3}$ & 1.9984014443252818(-15) \\ 
 \hline 
 \end{tabular} 
 \end{table}

\begin{example}
\label{exm2}
Suppose that $\beta=-1/3$ and M\"{u}ntz sequence $\Lambda_{2N+1}$ satisfies
\[
\lambda_{2k} = \lambda_{2k+1} = k-\frac{1}{2}, \quad k=0,1,\cdots,N.
\]
Gaussian nodes $x_k$ and weights $\omega_k$, $k=0,1,\cdots,N$, as shown in \cref{tab:quadrature2}, are exact for any function in $\hat{M}(\Lambda_{2N+1})$.

Due to repeated indices, $\hat{M}(\Lambda_{2N+1})$ contains logarithmic functions.
From \cref{tab:quadrature2}, it is evident that the Gaussian nodes are densely clustered around 0, making them well-suited for accurately handling integrands that exhibit rapid growth or decline near the left endpoint of the interval. Similar to \cref{exm1}, in \cref{tab:err2}, we provide a list of relative errors $R[\varphi]$ as defined in \cref{eq:relative_err}, from which we observe that $N+1$ quadrature rule $Q_N[\cdot]$ achieve exactness, within machine accuracy, for a collection of basis functions belonging to $\hat{M}(\Lambda_{2N+1})$.\end{example}

{\footnotesize
\begin{longtable}{|p{0.8cm}|p{3.4cm}|p{3.4cm}|}
\caption{\rm Gauss quadrature of \cref{exm2}.}
\label{tab:quadrature2}\\
\hline
$N+1$ & Nodes $x_k$  & Weights $\omega_k$  \\
\hline 
\endfirsthead
\multicolumn{3}{r}{\textit{Continued on next page}} \\
\endfoot
\endlastfoot
\ \ 20 & 1.7885486758102558(-8) & 1.1523469504263825(-4) \\ 
 & 7.1551616529140532(-5) & 5.8503872337450366(-3) \\ 
 & 8.3561615801303415(-4) & 1.6215461369396420(-2) \\ 
 & 3.7289725070694703(-3) & 2.9742159038797088(-2) \\ 
 & 1.0841226296519608(-2) & 4.5322220327996356(-2) \\ 
 & 2.4650360296325360(-2) & 6.1907685033178526(-2) \\ 
 & 4.7699277205001729(-2) & 7.8483315017420313(-2) \\ 
 & 8.2247917905893128(-2) & 9.4076870863819106(-2) \\ 
 & 1.2993765604008317(-1) & 1.0778485377987003(-1) \\ 
 & 1.9150387312943998(-1) & 1.1880283936542314(-1) \\ 
 & 2.6656819611889676(-1) & 1.2645525913197742(-1) \\ 
 & 3.5353427020110151(-1) & 1.3022161026435305(-1) \\ 
 & 4.4960088886379523(-1) & 1.2975708380221534(-1) \\ 
 & 5.5089479647303985(-1) & 1.2490623592648446(-1) \\ 
 & 6.5271364954315925(-1) & 1.1570882831668874(-1) \\ 
 & 7.4985864107441091(-1) & 1.0239741728426736(-1) \\ 
 & 8.3702725462981842(-1) & 8.5386716624398701(-2) \\ 
 & 9.0923046967895849(-1) & 6.5255288538808617(-2) \\ 
 & 9.6219650408072144(-1) & 4.2721655477618713(-2) \\ 
 & 9.9273163921294560(-1) & 1.8641580804642555(-2) \\ 
\hline
\hline 
 \ \ 40 & 1.1093514362142486(-9) & 1.8057444657681708(-5) \\ 
 & 4.4555423015369943(-6) & 9.2111663330328875(-4) \\ 
 & 5.2578791974959081(-5) & 2.5842500435975871(-3) \\ 
 & 2.3857885803177525(-4) & 4.8339802247692066(-3) \\ 
 & 7.0975130683830704(-4) & 7.5702889670536764(-3) \\ 
 & 1.6619379108385249(-3) & 1.0712736841684651(-2) \\ 
 & 3.3333984874852280(-3) & 1.4189619366383679(-2) \\ 
 & 5.9972993504626114(-3) & 1.7933242068639596(-2) \\ 
 & 9.9530127416244830(-3) & 2.1877894697564508(-2) \\ 
 & 1.5516454181706534(-2) & 2.5959008876182795(-2) \\ 
 & 2.3009710582646729(-2) & 3.0112906557680737(-2) \\ 
 & 3.2750231378357329(-2) & 3.4276871917926208(-2) \\ 
 & 4.5039867503105624(-2) & 3.8389410738086611(-2) \\ 
 & 6.0154048467500588(-2) & 4.2390620986094932(-2) \\ 
 & 7.8331385908669379(-2) & 4.6222627958035944(-2) \\ 
 & 9.9763982860287634(-2) & 4.9830053211846963(-2) \\ 
 & 1.2458871181630048(-1) & 5.3160495584589983(-2) \\ 
 & 1.5287970184084276(-1) & 5.6165008075362670(-2) \\ 
 & 1.8464224606269530(-1) & 5.8798557911169114(-2) \\ 
 & 2.1980830659385084(-1) & 6.1020459538421661(-2) \\ 
 & 2.5823375507081436(-1) & 6.2794772061574217(-2) \\ 
 & 2.9969744458992525(-1) & 6.4090654045891315(-2) \\ 
 & 3.4390216383989641(-1) & 6.4882669767053980(-2) \\ 
 & 3.9047747783555770(-1) & 6.5151042017447869(-2) \\ 
 & 4.3898441297273816(-1) & 6.4881847522119687(-2) \\ 
 & 4.8892189830990607(-1) & 6.4067151908927977(-2) \\ 
 & 5.3973483117197907(-1) & 6.2705082035888510(-2) \\ 
 & 5.9082359445178412(-1) & 6.0799834313980346(-2) \\ 
 & 6.4155481636620071(-1) & 5.8361618480053577(-2) \\ 
 & 6.9127313181855365(-1) & 5.5406537072951248(-2) \\ 
 & 7.3931367871632625(-1) & 5.1956401646260371(-2) \\ 
 & 7.8501504324459204(-1) & 4.8038487514455615(-2) \\ 
 & 8.2773235570215697(-1) & 4.3685229585270061(-2) \\ 
 & 8.6685023342696965(-1) & 3.8933862619944019(-2) \\ 
 & 9.0179526983967162(-1) & 3.3826010239928678(-2) \\ 
 & 9.3204777915725145(-1) & 2.8407228822337316(-2) \\ 
 & 9.5715252676017615(-1) & 2.2726518444260720(-2) \\ 
 & 9.7672821701962775(-1) & 1.6835845666401086(-2) \\ 
 & 9.9047567078037091(-1) & 1.0790014028473644(-2) \\ 
 & 9.9818651979791806(-1) & 4.6532712355365508(-3) \\ 
  \hline 
 \end{longtable}
 }

\begin{table}[htbp]
\centering \small
\caption{\rm Relative error of Gauss quadrature in \cref{exm2}.}  
\label{tab:err2}
\begin{tabular}{|c|c|c|c|c|}
\hline 
$N+1$ & $\varphi(x) $ & Relative error $R[\varphi]$ & $\varphi(x)$ & Relative error $R[\varphi]$ \\
\hline 
20 & $x^{-1/2}$  & 6.6613381477509392(-16) & $x^{-1/2} \log x$ & 1.3322676295501878(-15) \\ 
 & $x^{1/2}$ & 1.6653345369377348(-15) & $x^{1/2} \log x$ & 5.5511151231257827(-15) \\ 
 & $x^{3/2}$ & 4.4408920985006262(-16) & $x^{3/2} \log x$ & 8.8817841970012523(-16) \\ 
 & $x^{5/2}$ & 5.5511151231257827(-16) & $x^{5/2} \log x$ & 6.6613381477509392(-16) \\ 
 & $x^{7/2}$ & 5.5511151231257827(-16) & $x^{7/2} \log x$ & 9.9920072216264089(-16) \\ 
 & $x^{9/2}$ & 3.3306690738754696(-16) & $x^{9/2} \log x$ & 8.8817841970012523(-16) \\ 
 & $x^{11/2}$ & 2.2204460492503131(-16) & $x^{11/2} \log x$ & 1.3322676295501878(-15) \\ 
 & $x^{13/2}$ & 2.2204460492503131(-16) & $x^{13/2} \log x$ & 1.5543122344752192(-15) \\ 
 & $x^{15/2}$ & 6.6613381477509392(-16) & $x^{15/2} \log x$ & 2.1094237467877974(-15) \\ 
 & $x^{17/2}$ & 6.6613381477509392(-16) & $x^{17/2} \log x$ & 2.6645352591003757(-15) \\ 
 & $x^{19/2}$ & 8.8817841970012523(-16) & $x^{19/2} \log x$ & 2.8865798640254070(-15) \\ 
 & $x^{21/2}$ & 1.3322676295501878(-15) & $x^{21/2} \log x$ & 3.4416913763379853(-15) \\ 
 & $x^{23/2}$ & 1.7763568394002505(-15) & $x^{23/2} \log x$ & 3.7747582837255322(-15) \\ 
 & $x^{25/2}$ & 2.4424906541753444(-15) & $x^{25/2} \log x$ & 3.8857805861880479(-15) \\ 
 & $x^{27/2}$ & 2.6645352591003757(-15) & $x^{27/2} \log x$ & 4.2188474935755949(-15) \\ 
 & $x^{29/2}$ & 3.3306690738754696(-15) & $x^{29/2} \log x$ & 4.4408920985006262(-15) \\ 
 & $x^{31/2}$ & 3.7747582837255322(-15) & $x^{31/2} \log x$ & 4.7739590058881731(-15) \\ 
 & $x^{33/2}$ & 4.4408920985006262(-15) & $x^{33/2} \log x$ & 4.5519144009631418(-15) \\ 
  & $x^{35/2}$ & 5.1070259132757201(-15) & $x^{35/2} \log x$ & 4.5519144009631418(-15) \\ 
 & $x^{37/2}$ & 5.5511151231257827(-15) & $x^{37/2} \log x$ & 4.7739590058881731(-15) \\ 
 \hline 
 \end{tabular} 
 \end{table}

\begin{example}
In this example, we present an application of Gauss quadrature. Consider the integrand
\[
\psi(x) = \sin (4\pi x) + \frac{\log x (1-x)}{1+x}.
\]
Thus its exact integration value \cite{gradshteyn2014table} is
\begin{equation}
\label{integral1}
\int_0^1 \psi(x) \mathrm{d} x
= 
1 - \frac{\pi^2}{6}.
\end{equation}
Moreover, the integral of the Bessel function $J_0(x)$ is given by
\begin{equation}
\label{integral2}
\int_0^1 J_0(x) (1+\log x) \mathrm{d} x =  -0.0531080375895118730468486186978172 \cdots.
\end{equation}
For obtaining exactly the integrals in numerical, we choose $\beta=0$ and M\"{u}ntz sequences $\Lambda_{2N+1}$ that are categorized into three cases for a comparison.
\begin{itemize}
\item Case 1. $\lambda_k = k$ for $k=0,1,\cdots,2N+1$.
\item Case 2. $\lambda_k = \left \lfloor {k}/{2} \right \rfloor$ for $k=0,1,\cdots,2N+1$.
\item Case 3. $\lambda_k = \left \lfloor {k}/{3} \right \rfloor$ for $k=0,1,\cdots,2N+1$.
\end{itemize}
By \cref{repeated_indices}, it is evident that in case 1, the set $\hat{M}(\Lambda_{2N+1})$ exclusively consists of algebraic polynomials, that in case 2, it includes algebraic polynomials as well as algebraic polynomials multiplied by a logarithmic function, and that in case 3, the set comprises algebraic polynomials, algebraic polynomials multiplied by a logarithmic function, and algebraic polynomials multiplied by a squared logarithmic function. It is important to note that neither $\psi(x)$ nor $J_0(x)(1+\log x)$ belongs to $\hat{M}(\Lambda_{2N+1})$ in any of the three cases considered.

By employing \cref{alg:gauss_quadrature}, we can calculate the Gaussian nodes and weights required for the three quadrature rules $Q_N[\cdot]$ to approximate the integrals in \cref{integral1} and \cref{integral2}.
It is noteworthy that in case 1, the quadrature rule is equivalent to the classical Gauss-Legendre quadrature. The result of the approximations is depicted in \cref{fig:figure1}, where the term $Error$ represents the absolute error, and $N$ denotes the number of Gaussian nodes utilized.

In \cref{fig:figure1}, It can be observed that despite the singularity of both integrands $\psi(x)$ and $J_0(x)(1+\log x)$ at $0$, their integrals can be well approximated by the quadrature rules of case 2 and case 3, which is consistent with the result obtained in \cref{error_estimation}. However, classic Gauss-Legendre quadrature fails to provide accurate approximations.
\end{example}

\begin{figure}[htbp]
    \centering
    \caption{\rm A comparison of different Gauss quadrature rules.}
    \label{fig:figure1}
    \includegraphics[width=0.45\textwidth]{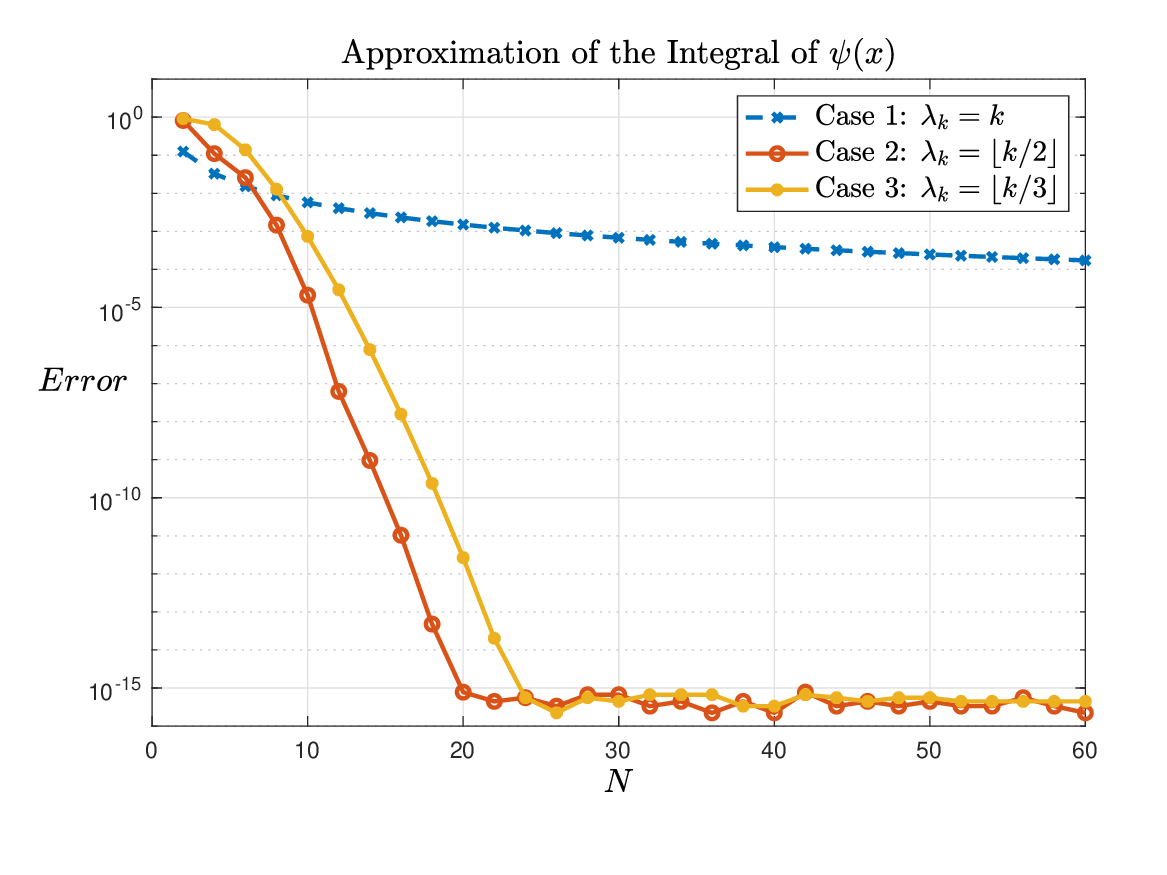}
    \hspace{0.02in}
    \includegraphics[width=0.45\textwidth]{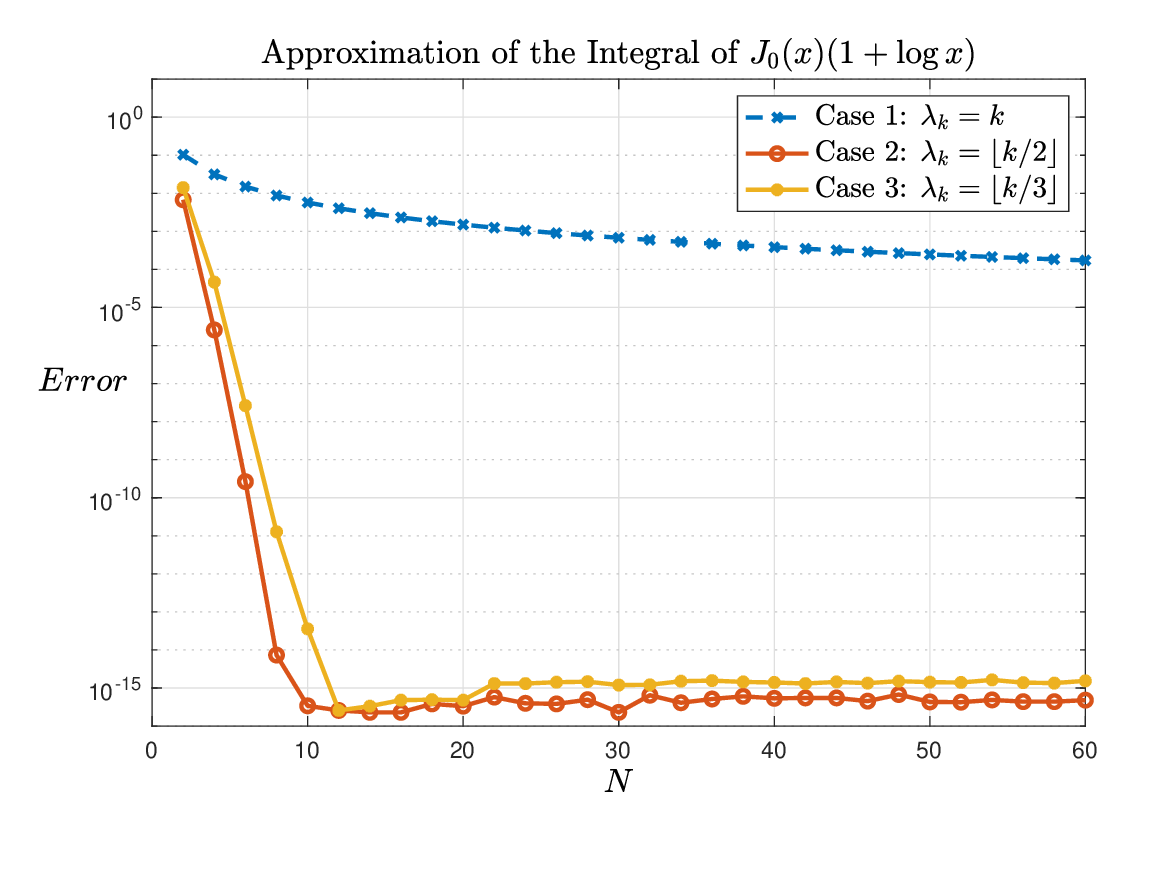}
\end{figure}

\bibliographystyle{siamplain}
\bibliography{quadrature}

\begin{thebibliography}{10}

\bibitem{allgower2012numerical}
{\sc E.~L. Allgower and K.~Georg}, {\em Numerical continuation methods: an
  introduction}, vol.~13, Springer Science \& Business Media, 2012.

\bibitem{almira2007muntz}
{\sc J.~M. Almira}, {\em M{\"u}ntz type theorems {I}}, arXiv preprint
  arXiv:0710.3570,  (2007).

\bibitem{borwein1994muntz}
{\sc P.~Borwein, T.~Erd{\'e}lyi, and J.~Zhang}, {\em {M{\"u}ntz} systems and
  orthogonal {M{\"u}ntz-Legendre} polynomials}, Transactions of the American
  Mathematical Society, 342 (1994), pp.~523--542.

\bibitem{rokhlin2010nonlinear}
{\sc J.~Bremer, Z.~Gimbutas, and V.~Rokhlin}, {\em A nonlinear optimization
  procedure for generalized {Gaussian} quadratures}, SIAM Journal on Scientific
  Computing, 32 (2010), pp.~1761--1788.

\bibitem{chen2021explicit}
{\sc P.~Chen and X.~Li}, {\em Explicit gaussian quadrature rules for c 1 cubic
  splines with non-uniform knot sequences}, Communications in Mathematics and
  Statistics, 9 (2021), pp.~331--345.

\bibitem{rokhlin1999nonlinear}
{\sc H.~Cheng, V.~Rokhlin, and N.~Yarvin}, {\em Nonlinear optimization,
  quadrature, and interpolation}, SIAM Journal on Optimization, 9 (1999),
  pp.~901--923.

\bibitem{cools2003}
{\sc R.~Cools and A.~Haegemans}, {\em Algorithm 824: {CUBPACK}: A package for
  automatic cubature; framework description}, ACM Transactions on Mathematical
  Software (TOMS), 29 (2003), pp.~287--296.

\bibitem{gautschi2011numerical}
{\sc W.~Gautschi}, {\em Numerical analysis}, Springer Science \& Business
  Media, 2011.

\bibitem{glaser2007fast}
{\sc A.~Glaser, X.~Liu, and V.~Rokhlin}, {\em A fast algorithm for the
  calculation of the roots of special functions}, SIAM Journal on Scientific
  Computing, 29 (2007), pp.~1420--1438.

\bibitem{golub1969calculation}
{\sc G.~H. Golub and J.~H. Welsch}, {\em Calculation of {Gauss} quadrature
  rules}, Mathematics of computation, 23 (1969), pp.~221--230.

\bibitem{gradshteyn2014table}
{\sc I.~S. Gradshteyn and I.~M. Ryzhik}, {\em Table of integrals, series, and
  products}, Academic press, 2014.

\bibitem{huybrechs2009generalized}
{\sc D.~Huybrechs and R.~Cools}, {\em On generalized {Gaussian} quadrature
  rules for singular and nearly singular integrals}, SIAM journal on numerical
  analysis, 47 (2009), pp.~719--739.

\bibitem{jorge2006}
{\sc N.~Jorge and J.~W. Stephen}, {\em Numerical optimization}, 2006.

\bibitem{joseph2013}
{\sc J.~R. JOSEPH}, {\em First Course in Abstract Algebra: with Applications},
  PRENTICE HALL PTR, 2013.

\bibitem{karlin1966}
{\sc S.~Karlin and W.~J. Studden}, {\em Tchebycheff systems: With applications
  in analysis and statistics}, vol.~15, Interscience Publishers, 1966.

\bibitem{lagarias1998convergence}
{\sc J.~C. Lagarias, J.~A. Reeds, M.~H. Wright, and P.~E. Wright}, {\em
  Convergence properties of the {Nelder--Mead} simplex method in low
  dimensions}, SIAM Journal on optimization, 9 (1998), pp.~112--147.

\bibitem{li2001}
{\sc Q.~Li}, {\em Numerical Analysis (in Chinese)}, Tsinghua University Press,
  2001.

\bibitem{lombardi2009}
{\sc G.~Lombardi}, {\em Design of quadrature rules for {M{\"u}ntz} and
  {M{\"u}ntz}-logarithmic polynomials using monomial transformation},
  International journal for numerical methods in engineering, 80 (2009),
  pp.~1687--1717.

\bibitem{ma1996generalized}
{\sc J.~Ma, V.~Rokhlin, and S.~Wandzura}, {\em Generalized {Gaussian}
  quadrature rules for systems of arbitrary functions}, SIAM Journal on
  Numerical Analysis, 33 (1996), pp.~971--996.

\bibitem{milovanovic1999muntz}
{\sc G.~V. Milovanovi{\'c}}, {\em M{\"u}ntz orthogonal polynomials and their
  numerical evaluation}, in Applications and Computation of Orthogonal
  Polynomials: Conference at the Mathematical Research Institute Oberwolfach,
  Germany March 22--28, 1998, Springer, 1999, pp.~179--194.

\bibitem{milovanovic2017computing}
{\sc G.~V. Milovanovi{\'c}}, {\em Computing integrals of highly oscillatory
  special functions using complex integration methods and {Gaussian}
  quadratures}, Dolomites Research Notes on Approximation, 10 (2017).

\bibitem{milovanovic2005gaussian}
{\sc G.~V. Milovanovic and A.~S. Cvetkovic}, {\em Gaussian-type quadrature
  rules for {M{\"u}ntz} systems}, SIAM Journal on Scientific Computing, 27
  (2005), pp.~893--913.

\bibitem{milovanovic2008trigonometric}
{\sc G.~V. Milovanovi{\'c}, A.~S. Cvetkovi{\'c}, and M.~P. Stani{\'c}}, {\em
  Trigonometric orthogonal systems and quadrature formulae}, Computers \&
  Mathematics with Applications, 56 (2008), pp.~2915--2931.

\bibitem{powell1981approximation}
{\sc M.~J.~D. Powell et~al.}, {\em Approximation theory and methods}, Cambridge
  university press, 1981.

\bibitem{rudin1976principles}
{\sc W.~Rudin et~al.}, {\em Principles of mathematical analysis}, vol.~3,
  McGraw-hill New York, 1976.

\bibitem{schwab1994}
{\sc C.~Schwab}, {\em Variable order composite quadrature of singular and
  nearly singular integrals}, Computing, 53 (1994), pp.~173--194.

\bibitem{stein2010complex}
{\sc E.~M. Stein and R.~Shakarchi}, {\em Complex analysis}, vol.~2, Princeton
  University Press, 2010.

\bibitem{taslakyan1984}
{\sc A.~Taslakyan}, {\em Some properties of {Legendre} quasi-polynomials with
  respect to a {M{\"u}ntz} system}, Mathematics, 2 (1984), pp.~179--189.

\bibitem{trefethen2022numerical}
{\sc L.~N. Trefethen and D.~Bau}, {\em Numerical linear algebra}, vol.~181,
  Siam, 2022.

\bibitem{rokhlin1998generalized}
{\sc N.~Yarvin and V.~Rokhlin}, {\em Generalized gaussian quadratures and
  singular value decompositions of integral operators}, SIAM Journal on
  Scientific Computing, 20 (1998), pp.~699--718.

\end{thebibliography}
\end{document}